\documentclass[reqno]{amsart}

\usepackage{
    enumitem, %
    amsmath, %
    amssymb, %
    bm,
    mathdots,
    url, %
}

\usepackage{amsthm}
\theoremstyle{plain}
\newtheorem{theorem}{Theorem}[section]
\newtheorem{proposition}[theorem]{Proposition}
\newtheorem{lemma}[theorem]{Lemma}
\newtheorem{corollary}[theorem]{Corollary}

\theoremstyle{remark}
\newtheorem{remark}[theorem]{Remark}
\newtheorem{example}[theorem]{Example}

\numberwithin{equation}{section}

\usepackage[colorlinks=true]{hyperref}

\allowdisplaybreaks[4]

\usepackage{CJKutf8}

\title[4-Dimensional Isoparametric Hypersurfaces of Index 2]{4-Dimensional Isoparametric Hypersurfaces of Index 2 in the Pseudo-Riemannian Space Forms}
\author[Y. Sasahara]{Yuta Sasahara}
\thanks{This work was supported by JST SPRING, Grant Number \textsf{JPMJSP2156}.}
\address{Department of Mathematical Sciences, Graduate School of Science, Tokyo Metropolitan University, Minami-Osawa 1-1, Hachioji, Tokyo, 192-0397, Japan}
\email{sasahara-yuta@ed.tmu.ac.jp, ysasahara0804+res@gmail.com}

\begin{document}
\begin{abstract}
    We study isoparametric hypersurfaces, whose principal curvatures are all constant, in the pseudo-Riemannian space forms.
    In this paper, we investigate three topics.
    Firstly, according to Petrov's classification theorem, we give a classification of hypersurfaces of index 2 with respect to a pair of a shape operator and a metric.
    Therefore, we can define types of isoparametric hypersurfaces of index 2 concerning the classification.
    Secondly, we give several examples of certain types.
    Thirdly, we show that there exist no isoparametric hypersurfaces of index 2 whose shape operators have complex principal curvatures in certain cases.
\end{abstract}

\maketitle

\section{Introduction}\label{sec:1}
Hahn gave the definition of isoparametric hypersurfaces: a non-degenerate hypersurface in the pseudo-Riemannian space form is isoparametric if its principal curvatures (the eigenvalue functions of its shape operator) are all constant, which is written in \cite[2.1. Proposition]{MR753432}.
After \cite{MR753432}, classification problems of Lorentzian isoparametric hypersurfaces have been developed in \cite{MR783023}, \cite{MR1696128}, and \cite{MR3077208}.
Since the study of hypersurfaces in the pseudo-Riemannian space forms is generally much more difficult than that in the Riemannian space forms, the study of isoparametric hypersurfaces in the pseudo-Riemannian space forms is still developing.

In \cite{MR783023}, \cite{MR1696128}, and \cite{MR3077208}, they used the other definition of isoparametric hypersurfaces: a non-degenerate hypersurface in the pseudo-Riemannian space form is isoparametric if the minimal polynomial of its shape operator is constant.
If a hypersurface satisfies the latter definition, then it also satisfies the former definition.
However, the converse of this claim is not true.
In fact, \cite[2.5. \textit{Example}]{MR753432} is a counterexample to the claim, which is explained in Example \ref{exam:0-1}.
The shape operator of an isoparametric hypersurface as the latter definition of index 1 can be put into exactly one of the four types, which are obtained by Petrov's classification theorem in \cite{MR0244912}.
It is explained in Sections \ref{sec:3-1} and \ref{sec:3-2}.
The fact helps us solve classification problems of isoparametric hypersurfaces as the latter definition.
In this paper, we adopt the former definition and consider a classification of isoparametric hypersurfaces of index 2 with respect to a pair of a shape operator and a metric.
Therefore, we can define 13 types of isoparametric hypersurfaces concerning the classification, which are called Petrov types in this paper, in Section \ref{sec:3-3}.

We have investigated the inhomogeneity of isoparametric hypersurfaces of OT-FKM-type in the pseudo-sphere in \cite{sasahara2023inhomogeneity}.
Next, we focus on the notion of indices of hypersurfaces in the pseudo-Riemannian space forms.
The purpose of this paper is to investigate the existence and the non-existence of 4-dimensional isoparametric hypersurfaces of index 2 in the pseudo-Riemannian space forms.

Firstly, we give several examples of isoparametric hypersurfaces of index 2 in the pseudo-Riemannian space forms.
We find the appropriate matrices and obtain the examples by substituting them into the matrices in \cite[\S 2, Quadratic Examples]{MR753432} in Section \ref{sec:4-1}.
Referring to local examples of isoparametric hypersurfaces of index 1 which are written in \cite[3. Examples]{MR783023}, we obtain several examples of ones of index 2 in Section \ref{sec:4-2}.

Secondly, we show the non-existence of isoparametric hypersurfaces whose Petrov types are certain ones of index 2 in Section \ref{sec:4-3}.
In the previous research, we have already found that there are no Lorentzian isoparametric hypersurfaces whose shape operator has complex principal curvatures in the Minkowski space and the pseudo-sphere in \cite[Theorem 4.10]{MR783023} and \cite[Theorem 4.1]{MR3077208}, respectively.
On the analogy of it, we show that there are no isoparametric hypersurfaces of index 2 whose shape operators have complex principal curvatures in certain cases.

\section{Preliminaries}\label{sec:2}
In this section, we prepare our notations and state basic facts about isoparametric hypersurfaces.
Throughout this paper, we suppose that all of manifolds and mappings are smooth.

We write ${}^{t}{\!}P$ for the transpose of a matrix $P$.
The identity matrix and the zero matrix of order $n$ are denoted by $E_{n}$ and $O_{n}$, respectively.
The column vector $\bm{e}_{i}\in \mathbb{R}^{n}$ denotes the $i$-th column of $E_{n}$.
We write $P\oplus Q$ for the direct sum of matrices $P$ and $Q$.
The anti-diagonal matrix whose anti-diagonal entries are all $1$ of order $n$ is denoted by $E^{\mathfrak{a}}_{n}$.
Let $V$ be a finite-dimensional vector space over $\mathbb{R}$ equipped with a non-degenerate symmetric bilinear form $B\colon V\times V\to \mathbb{R}$.
The set of all endomorphisms of $V$ which are self-adjoint with respect to $B$ is denoted by $\mathrm{Sym}(V, B)$.

Let $n$ be a positive integer except $1$.
Let $s\in \{0, 1, \ldots, n\}$.
Let $\mathbb{R}^{n}_{s}$ be an $n$-dimensional pseudo-Euclidean space provided with the pseudo-inner product $\langle u, v\rangle_{s}:={}^{t}{\!}u ((-E_{s})\oplus E_{n-s})v$ for $u, v\in \mathbb{R}^{n}_{s}$.
We define the pseudo-sphere $S^{n}_{s}:=\{x\in \mathbb{R}^{n+1}_{s}\mid \langle x, x\rangle_{s}=1\}$ and the pseudo-hyperbolic space $H^{n}_{s}:=\{x\in \mathbb{R}^{n+1}_{s+1}\mid \langle x, x\rangle_{s+1}=-1\}$ of index $s$.
(However, we define $S^{n}_{n}=\{x\in \mathbb{R}^{n+1}_{n}\mid \langle x, x\rangle_{s}=1, x^{n+1}>0\}$ and $H^{n}_{0}:=\{x\in \mathbb{R}^{n+1}_{1}\mid \langle x, x\rangle_{1}=-1, x^{1}>0\}$.)
The spaces $S^{n}_{s}, \mathbb{R}^{n}_{s}$, and $H^{n}_{s}$ are called the pseudo-Riemannian space forms, which have the constant curvatures $\kappa=1$, $0$, and $-1$, respectively.
The $n$-dimensional pseudo-Riemannian space form of index $s$ with constant curvature $\kappa$ is denoted by $N^{n}_{s}(\kappa)$.
Let $\bar{\nabla}$ be the natural Levi-Civita connection of $N(\kappa)$.

Let $n$ be a positive integer and $s\in \{0, 1, \ldots, n+1\}$.
Let $M\subset N^{n+1}_{s}(\kappa)$ be a non-degenerate connected hypersurface and suppose that there is a unit normal vector field $\xi$ on $M$.
Let $\iota\colon M\hookrightarrow N(\kappa)$ be the inclusion map.
The induced metric and the induced connection of $M$ are denoted by $\langle \cdot, \cdot\rangle$ and $\nabla$, respectively.
Let $\nu:=\langle \xi, \xi\rangle$ and $\delta:=\kappa\cdot\nu\in \{-1, 0, 1\}$.
We call $\delta$ the type of $M$.

We define $I(X):=X$ and $A(X):=-\bar{\nabla}^{\iota^{*}TN(\kappa)}_{X}\xi$ for $X\in \Gamma(TM)$.
Then $I, A\in \Gamma(\mathrm{End}(TM))$ holds.
The endomorphism $I$ and $A$ are called the identity transformation of $TM$ and the shape operator of $M$, respectively.
For each $x\in M$, $A_{x}\in \mathrm{Sym}(T_{x}M, \langle\cdot,\cdot\rangle_{x})$ holds.
Referring to \cite[2.1]{MR753432}, $M$ is called isoparametric if the principal curvatures of $M$ (the eigenvalue functions of $A$) are all constant on $M$.
Let $M\subset N^{n+1}_{s}(-1)=H^{n+1}_{s}$ be an isoparametric hypersurface of type $\delta$.
By considering the anti-isometry, that is, by replacing $\langle \cdot, \cdot\rangle$ with $-\langle\cdot, \cdot\rangle$, we obtain the isoparametric hypersurface $(M, -\langle\cdot, \cdot\rangle)\subset N^{n+1}_{n+1-s}(1)=S^{n+1}_{n+1-s}$ of type $\delta$.
Thus, it is sufficient that we study isoparametric hypersurfaces in the pseudo-Riemannian space forms of constant curvature $\kappa=0, 1$.

A function $f\colon N^{n+1}_{s}(\kappa)\to \mathbb{R}$ is isoparametric if there exist functions $\Phi, \Psi\colon \mathbb{R}\to \mathbb{R}$ such that $\langle \mathrm{grad}^{N}(f), \mathrm{grad}^{N}(f)\rangle=\Phi\circ f$ and $\Delta^{N}(f)=\Psi\circ f$ hold, where $\mathrm{grad}^{N}(f)$ and $\Delta^{N}(f)$ are the gradient and the Laplacian of $f$ in $N^{n+1}_{s}(\kappa)$, respectively.
For an isoparametric function $f$, we set
$
    \mathrm{W}_{\mathrm{RN}}(f):=
    \left\{
        c\in f\left(N^{n+1}_{s}(\kappa)\right)
        \ \middle|\
        \Phi(c)\neq 0
    \right\}
$.
By \cite[2.2]{MR753432}, if $f\colon N^{n+1}_{s}(\kappa)\to \mathbb{R}$ is isoparametric and $c\in \mathrm{W}_{\mathrm{RN}}(f)$, then each connected component $M$ of $f^{-1}(c)\subset N^{n+1}_{s}(\kappa)$ is an isoparametric hypersurface.
Then $\mathrm{grad}^{N}(f)/\sqrt{\left|\Phi(c)\right|}$ is a unit normal vector field on $M$.

Let $R$ be the curvature tensor of $M$.
Since $\nabla$ is torsion-free, we have
\begin{equation}
    \langle R(X, Y)Z, W\rangle
    =
    \langle\nabla_{X}\nabla_{Y}Z, W\rangle
    -
    \langle\nabla_{Y}\nabla_{X}Z, W\rangle
    -
    \langle\nabla_{\nabla_{X}Y}Z, W\rangle
    +
    \langle\nabla_{\nabla_{Y}X}Z, W\rangle
    ,
    \label{eq:R}
\end{equation}
for $X, Y, Z, W\in \Gamma(TM)$.
On the other hand, we have the following equations:
\begin{align}
    &R(X, Y)Z
    =
    \kappa\{\langle Y, Z\rangle X -\langle X, Z\rangle Y\}
    +
    \nu\{\langle AY, Z\rangle AX -\langle AX, Z\rangle AY\},\label{eq:1}\\
    &\nabla_{X}(AY)-A(\nabla_{X}Y)
    =
    \nabla_{Y}(AX)-A(\nabla_{Y}X),\label{eq:2}
\end{align}
for $X, Y, Z\in \Gamma(TM)$, which are called the Gauss equation and the Codazzi equation, respectively.
Taking the inner product of both sides of \eqref{eq:1} and \eqref{eq:2}, we have
\begin{equation}
    \begin{split}
        \langle R(X, Y)Z, W\rangle
        &=
        \kappa\{\langle Y, Z\rangle \langle X, W\rangle  -\langle X, Z\rangle \langle Y, W\rangle\}\\
        &+
        \nu\{\langle AY, Z\rangle \langle AX, W\rangle -\langle AX, Z\rangle \langle AY, W\rangle\},
    \end{split}
    \label{eq:Gauss}
\end{equation}
\begin{equation}
    \langle \nabla_{X}(AY), Z\rangle -\langle \nabla_{X}Y, AZ\rangle
    =
    \langle \nabla_{Y}(AX), Z\rangle -\langle \nabla_{Y}X, AZ\rangle
    ,
    \label{eq:Codazzi}
\end{equation}
for $X, Y, Z, W\in \Gamma(TM)$.
In this paper, \eqref{eq:Gauss} is also called the Gauss equation and \eqref{eq:Codazzi} is denoted by $\{X, Y\}Z$.

We have Cartan's identity by \cite[Corollary 4.3]{MR2443486}.
Let $k_{1}, k_{2}, \ldots, k_{p}$ be the distinct principal curvatures of an isoparametric hypersurface $M$ ($p\geqq 2$) and $m_{1}, m_{2}, \ldots, m_{p}$ their algebraic multiplicities.
If $k_{i}$ is real and the geometric multiplicity of $k_{i}$ is equal to $m_{i}$, then we have the equation
\begin{equation}
    \sum_{j\in \{1, 2, \ldots, p\}\setminus\{i\}}
    m_{j}\frac{\delta+k_{i}k_{j}}{k_{j}-k_{i}}
    =
    0,
    \label{eq:Cartan}
\end{equation}
which is called Cartan's formula (Cartan's identity) for $k_{i}$.

\section{A Shape Operator}\label{sec:3}
In this section, we state Petrov's classification theorem and give a classification of isoparametric hypersurfaces of index 2 with respect to a pair of a shape operator and a metric.

\subsection{Petrov's Classification Theorem}\label{sec:3-1}
Petrov's classification theorem is a claim regarding linear algebra and is written in \cite[{\S 9. The Principal Axes Theorem for a Tensor}]{MR0244912}.
(It is also written in \cite[3.10]{koike-riron} in Japanese.)
Since the Jordan normal form theorem underlies this claim, we first begin with it.

Let $V$ be an $n$-dimensional real vector space.
Let $A\in \mathrm{End}(V)$ throughout Section \ref{sec:3-1}.
We define $V^{\mathbb{C}}$ and $A^{\mathbb{C}}$ to be the complexification of $V$ and $A$, respectively.
We note that, if $\lambda\in \mathbb{C}$ is an eigenvalue of $A^{\mathbb{C}}$, the complex conjugate $\overline{\lambda}$ is also an eigenvalue of $A^{\mathbb{C}}$.
The characteristic polynomial of $A^{\mathbb{C}}$ is denoted by $\varphi_{A^{\mathbb{C}}}$ and can be expressed as the following form:
\begin{equation*}
    \varphi_{A^{\mathbb{C}}}(t)
    =
    \prod_{i=1}^{a}(t-\lambda_{i})^{\mu_{i}}
    \cdot
    \prod_{j=1}^{b}
    (t-(\alpha_{j}+\beta_{j}\sqrt{-1}))^{\nu_{j}}
    (t-(\alpha_{j}-\beta_{j}\sqrt{-1}))^{\nu_{j}}
    ,
\end{equation*}
where $a$ and $b$ are non-negative integer, and where $\lambda_{i}, \alpha_{j}, \beta_{j}\in \mathbb{R}$; $\beta_{j}$ is positive for all $j$; $\lambda_{1}, \ldots, \lambda_{a}$ are the distinct real eigenvalues of $A^{\mathbb{C}}$; $\alpha_{1}+\beta_{1}\sqrt{-1}, \ldots, \alpha_{b}+\beta_{b}\sqrt{-1}$ are the distinct complex eigenvalues of $A^{\mathbb{C}}$ whose imaginary part is positive, and where $\{\mu_{i}\}_{i=1}^{a}$ and $\{\nu_{j}\}_{j=1}^{b}$ satisfy
\begin{equation*}
    \sum_{i=1}^{a}\mu_{i}
    +
    \sum_{j=1}^{b}2\nu_{j}
    =
    n
    .
\end{equation*}
The Jordan block of the diagonal element $\lambda\in \mathbb{C}$ and size $m$ is defined by
\begin{equation*}
    J_{m}(\lambda)
    :=
    \begin{pmatrix}
        \lambda & 1 & 0& \cdots & 0\\
        0 & \lambda & 1 &\ddots &\vdots\\
        \vdots & \ddots &\ddots &\ddots &0 \\
        \vdots &  & \ddots & \ddots& 1\\
        0&\cdots&\cdots&0&\lambda
    \end{pmatrix}
    .
\end{equation*}
\begin{proposition}[The Jordan Normal Form Theorem]
    \label{prop:Jordan}
    There exist the family $\{p_{i, k, u}\}$ of vectors of $V$ and the family $\{q_{j, l, v}\}$ of vectors of $V^{\mathbb{C}}$ such that
    \begin{equation*}
        \mathcal{B}:
        \left\{
            \left\{
                \left\{
                    p_{i, k, u}
                \right\}_{u=1}^{m_{i, k}}
            \right\}_{k=1}^{s_{i}}
        \right\}_{i=1}^{a}
        \cup
        \left\{
            \left\{
                \left\{
                    q_{j, l, v}
                \right\}_{v=1}^{n_{j, l}}
            \right\}_{l=1}^{t_{j}}
            \cup
            \left\{
                \left\{
                    \overline{q_{j, l, v}}
                \right\}_{v=1}^{n_{j, l}}
            \right\}_{l=1}^{t_{j}}
        \right\}_{j=1}^{b}
    \end{equation*}
    is an ordered basis of $V^{\mathbb{C}}$ and then the matrix representation of $A^{\mathbb{C}}$ with respect to $\mathcal{B}$ is
    \begin{align*}
        J:=
        \left\{
            \bigoplus_{i=1}^{a}
            \left(
                \bigoplus_{k=1}^{s_{i}}
                J_{m_{i, k}}(\lambda_{i})
            \right)
        \right\}
        \oplus
        \left[
            \bigoplus_{j=1}^{b}
            \left\{
                \left(
                    \bigoplus_{l=1}^{t_{j}}
                    J_{n_{j, l}}(\alpha_{j}+\beta_{j}\sqrt{-1})
                \right)
                \oplus
                \left(
                    \bigoplus_{l=1}^{t_{j}}
                    J_{n_{j, l}}(\alpha_{j}-\beta_{j}\sqrt{-1})
                \right)
            \right\}
        \right]
        ,
    \end{align*}
    where $s_{i}$ is a positive integer for all $i\in \{1,\ldots, a\}$, and where the sequence of positive integers $\{m_{i, k}\}_{k=1}^{s_{i}}$ satisfies
    \begin{equation*}
        1\leqq m_{i, 1}\leqq m_{i, 2}\leqq \cdots \leqq m_{i, s_{i}}\leqq \mu_{i}
        ,\quad
        \mu_{i}=\sum_{k=1}^{s_{i}}m_{i, k},
    \end{equation*}
    and where $t_{j}$ is a positive integer for all $j\in \{1, \ldots, b\}$, and where the sequence of positive integers $\{n_{j, l}\}_{l=1}^{t_{j}}$ satisfies
    \begin{equation*}
        1\leqq n_{j, 1}\leqq n_{j, 2}\leqq \cdots \leqq n_{j, t_{j}}\leqq \nu_{j}
        ,\quad
        \nu_{j}=\sum_{l=1}^{t_{j}}n_{j, l}.
    \end{equation*}
    Here, $s_{i}, \{m_{i, k}\}, t_{j}, \{n_{j, l}\}$ are uniquely determined by $A$.
    (The basis $\mathcal{B}$ is written in ascending order and the positions of the vectors of $\mathcal{B}$ correspond to those of the Jordan blocks of $J$.)
\end{proposition}
The matrix $J$ in Proposition \ref{prop:Jordan} is called the Jordan normal form of $A$.
Under the above setting, we now claim Petrov's classification theorem.
\begin{theorem}[Petrov's Classification Theorem]
    \label{theo:Petrov}
    Let $B$ be a non-degenerate symmetric bilinear form on $V$.
    Suppose that $A\in \mathrm{Sym}(V, B)$.
    Then there is a basis $\{e_{i}\}$ of $V$ such that the matrix representation $(A^{i}_{j})$ of $A$ and the matrix $(B(e_{i}, e_{j}))$ of $B$ with respect to $\{e_{i}\}$ are the following forms:
    \begin{align*}
        &(A^{i}_{j})
        =
        \left\{
            \bigoplus_{i=1}^{a}
            \left(
                \bigoplus_{k=1}^{s_{i}}
                J_{m_{i, k}}(\lambda_{i})
            \right)
        \right\}
        \oplus
        \left\{
            \bigoplus_{j=1}^{b}
            \left(
                \bigoplus_{l=1}^{t_{j}}
                C_{n_{j, l}}(j)
            \right)
        \right\},\\
        &(B(e_{i}, e_{j}))
        =
        \left\{
            \bigoplus_{i=1}^{a}
            \left(
                \bigoplus_{k=1}^{s_{i}}
                \mathcal{E}_{i, k}
            \right)
        \right\}
        \oplus
        \left\{
            \bigoplus_{j=1}^{b}
            \left(
                \bigoplus_{l=1}^{t_{j}}
                \widehat{\mathcal{E}}_{j, l}
            \right)
        \right\}
        ,
    \end{align*}
    where the above symbols $\mathcal{E}_{i, k}, C_{n_{j, l}}(j), \widehat{\mathcal{E}}_{j, l}$ are defined by the following matrices:
    \begin{align*}
        \mathcal{E}_{i, k}
        &:=
        \epsilon(i, k)E^{\mathfrak{a}}_{m_{i, k}}
        \in \mathbb{R}^{m_{i, k}\times m_{i, k}}
        ,\quad (\epsilon(i, k)\in \{-1, 1\}),\\
        C(j)
        &:=
        \begin{pmatrix}
            \alpha_{j}&-\beta_{j}\\
            \beta_{j}&\alpha_{j}
        \end{pmatrix}
        ,\quad
        C_{n_{j, l}}(j)
        :=
        \begin{bmatrix}
            C(j) & E_{2} & O_{2}& \cdots & O_{2}\\
            O_{2} & C(j) & E_{2} &\ddots &\vdots\\
            \vdots & \ddots &\ddots &\ddots &O_{2} \\
            \vdots &  & \ddots & \ddots& E_{2}\\
            O_{2}&\cdots&\cdots&O_{2}&C(j)
        \end{bmatrix}
        \in \mathbb{R}^{2n_{j, l}\times 2n_{j, l}}
        ,\\
        \mathfrak{E}
        &:=
        \begin{pmatrix}
            -1 & 0\\
            0 & 1
        \end{pmatrix}
        ,\quad
        \widehat{\mathcal{E}}_{j, l}
        :=
        \begin{bmatrix}
            O_{2}&\cdots&O_{2}&\mathfrak{E}\\
            \vdots&\iddots&\mathfrak{E}&O_{2}\\
            O_{2}&\iddots&\iddots&\vdots\\
            \mathfrak{E}&O_{2}&\cdots&O_{2}
        \end{bmatrix}
        \in \mathbb{R}^{2n_{j, l}\times 2n_{j, l}}.
    \end{align*}
    If $m_{i, k_{1}}=m_{i, k_{2}}$ for $k_{1}, k_{2}\in \{1, 2, \ldots, s_{i}\}$, we suppose that $\epsilon(i, k_{1})\geqq \epsilon(i, k_{2})$.
    Here, $A$ determines whether $\epsilon(i, k)$ is $1$ or $-1$ for $(i, k)$ uniquely.
\end{theorem}
\begin{proof}[(Outline of the Proof of Theorem \ref{theo:Petrov})]
    We can take such a basis composed of Jordan chains by changing $\mathcal{B}$ in Proposition \ref{prop:Jordan} for each $i$ and $j$.
\end{proof}
In this paper, the pair $((A^{i}_{j}), (B(e_{i}, e_{j})))$ in Theorem \ref{theo:Petrov} is called the Petrov normal form of $(A, B)$.
The following Lemma \ref{lemm:1} is useful to compute the negative index of inertia of $B$.
\begin{lemma}
    \label{lemm:1}
    The number of negative eigenvalues of $\mathcal{E}_{i, k}$ for $(i, k)$ is
    \begin{equation*}
        \frac{1}{2}
        \left\{
            m_{i, k}
            +
            \frac{(-1)^{m_{i, k}}-1}{2}\epsilon(i, k)
        \right\}
        .
    \end{equation*}
    The number of negative eigenvalues of $\widehat{\mathcal{E}}_{j, l}$ for $(j, l)$ is $n_{j, l}$.
    Moreover, the negative index of inertia of $B$ is the sum of the numbers of negative eigenvalues of $\mathcal{E}_{i, k}$ and $\widehat{\mathcal{E}}_{j, l}$.
\end{lemma}
\begin{proof}
    Lemma \ref{lemm:1} follows from direct computation.
\end{proof}

\subsection{A Classification of Isoparametric Hypersurfaces of Index 1}\label{sec:3-2}
Let $M$ be an $n$-dimensional non-degenerate connected hypersurface of index 1 in the pseudo-Riemannian space form.
Let $A$ be the shape operator of $M$.
Since $A_{x}\in \mathrm{Sym}(T_{x}M, \langle\cdot,\cdot\rangle_{x})$, we have the Petrov normal form of $(A_{x}, \langle\cdot, \cdot\rangle_{x})$ by Theorem \ref{theo:Petrov}.
By Lemma \ref{lemm:1}, $(A_{x}, \langle\cdot, \cdot\rangle_{x})$ can be put into one of the following 5 types algebraically:
\begin{align*}
    &\text{I} :&
    A_{x}&\sim
    \begin{bmatrix}
        a_{0}&&&\\
        &a_{1}&&\\
        &&\ddots&\\
        &&&a_{n-1}
    \end{bmatrix},
    &&\langle\cdot,\cdot\rangle_{x}\sim
    \begin{bmatrix}
        -1&&&\\
        &1&&\\
        &&\ddots&\\
        &&&1
    \end{bmatrix},\\
    &\text{II} :&
    A_{x}&\sim
    \begin{bmatrix}
        b&1&&&\\
        0&b&&&\\
        &&a_{1}&&\\
        &&&\ddots&\\
        &&&&a_{n-2}
    \end{bmatrix},
    &&\langle\cdot,\cdot\rangle_{x}\sim
    \begin{bmatrix}
        0&\varepsilon&&&\\
        \varepsilon&0&&&\\
        &&1&&\\
        &&&\ddots&\\
        &&&&1
    \end{bmatrix},
    \\
    &\text{III} :&
    A_{x}&\sim
    \begin{bmatrix}
        b&1&0&&&\\
        0&b&1&&&\\
        0&0&b&&&\\
        &&&a_{1}&&\\
        &&&&\ddots&\\
        &&&&&a_{n-3}\\
    \end{bmatrix},
    &&\langle\cdot,\cdot\rangle_{x}\sim
    \begin{bmatrix}
        0&0&1&&&\\
        0&1&0&&&\\
        1&0&0&&&\\
        &&&1&&\\
        &&&&\ddots&\\
        &&&&&1\\
    \end{bmatrix},
    \\
    &\text{IV} :&
    A_{x}&\sim
    \begin{bmatrix}
        \alpha&-\beta&&&\\
        \beta&\alpha&&&\\
        &&a_{1}&&\\
        &&&\ddots&\\
        &&&&a_{n-2}
    \end{bmatrix},
    &&\langle\cdot,\cdot\rangle_{x}\sim
    \begin{bmatrix}
        -1&0&&&\\
        0&1&&&\\
        &&1&&\\
        &&&\ddots&\\
        &&&&1
    \end{bmatrix},
\end{align*}
where $a_{1}, a_{2}, \ldots, a_{n-1}, a_{0}, b ,\alpha, \beta\in \mathbb{R}$ and $\beta>0$, and where $\varepsilon\in \{-1, 1\}$.
We note that we distinguish between type II where $\varepsilon=1$ and type II where $\varepsilon=-1$ algebraically.
The above forms are different from the forms in Theorem \ref{theo:Petrov} since we rearrange the bases in Theorem \ref{theo:Petrov} such that the part concerning indices is located at the top position in the bases for convenience of explanation.
Although the above forms are slightly different from \cite[p.\,166--167]{MR783023}, we can modify them.
Let the left bases in \eqref{eq:bases1} be bases of the above forms.
\begin{equation}
    \begin{array}{lclcl}
        \text{II } (\varepsilon=1)&\text{:}&X_{1}, X_{2}, Y_{1}, \ldots, Y_{n-2}&\xrightarrow{\xi\to-\xi}&-X_{2}, X_{1}, Y_{1}, \ldots, Y_{n-2},\\
        \text{II } (\varepsilon=-1)&\text{:}&X_{1}, X_{2}, Y_{1}, \ldots, Y_{n-2}&\longrightarrow&X_{2}, X_{1}, Y_{1}, \ldots, Y_{n-2},\\
        \text{III}&\text{:}&X_{1}, X_{2}, X_{3}, Y_{1}, \ldots, Y_{n-3}&\longrightarrow&-X_{3}, X_{1}, X_{2}, Y_{1}, \ldots, Y_{n-3}.
    \end{array}
    \label{eq:bases1}
\end{equation}
The right bases in \eqref{eq:bases1} are bases of the forms in \cite[p.\,166--167]{MR783023}.
By considering the direction of the unit normal vectors at $x$, we can reduce the above 5 types to 4 types geometrically.
Since Theorem \ref{theo:Petrov} is a pointwise claim, $(A, \langle\cdot,\cdot\rangle)$ can have multiple types on $M$.
One of the examples is Example \ref{exam:0-1}.
\begin{example}\label{exam:0-1}
    There exists an isoparametric hypersurface whose $(A, \langle\cdot,\cdot\rangle)$ has 2 types.
    We refer to \cite[{2.5. \textit{Example}}]{MR753432}.
    Let
    \begin{equation*}
        \bm{a}_{1}
        :=
        \frac{\bm{e}_{1}+\bm{e}_{2}}{\sqrt{2}}
        ,
        \quad
        \bm{a}_{2}
        :=
        \frac{\bm{e}_{1}-\bm{e}_{2}}{\sqrt{2}}
        ,
    \end{equation*}
    where $\bm{e}_{1}, \bm{e}_{2}\in \mathbb{R}^{3}$.
    Let $f : \mathbb{R}^{2}\to \mathbb{R}^{3}_{1}$ be the mapping
    \begin{equation*}
        f(\bm{p}):=u\bm{a}_{1}+v\bm{a}_{2}-(\sin{v})\bm{e}_{3},\quad \bm{p}=(u, v)\in \mathbb{R}^{2},
    \end{equation*}
    where $\bm{e}_{3}\in \mathbb{R}^{3}$.
    Let $M:=f(\mathbb{R}^{2})\subset \mathbb{R}^{3}_{1}$.
    The shape operator $A_{f(\bm{p})}$ at $\bm{p}\in \mathbb{R}^{2}$ satisfies
    \begin{equation*}
        \xi_{f(\bm{p})}
        =
        \begin{pmatrix}
            {\displaystyle \frac{\cos{v}}{\sqrt{2}}}\vspace{1mm}\\
            {\displaystyle \frac{\cos{v}}{\sqrt{2}}}\\
            -1
        \end{pmatrix}
        ,\quad
        \mathcal{B}_{\bm{p}}
        :=
        \begin{bmatrix}
            {\displaystyle\frac{\partial f}{\partial u}(\bm{p})}&{\displaystyle\frac{\partial f}{\partial v}(\bm{p})}
        \end{bmatrix}
        ,\quad
        A_{f(\bm{p})}\mathcal{B}_{\bm{p}}
        =
        \mathcal{B}_{\bm{p}}
        \begin{pmatrix}
            0 & \sin{v}\\
            0 & 0
        \end{pmatrix}
        .
    \end{equation*}
    Thus, $M$ is isoparametric.
    By Table \ref{tab:3}, the type of $(A, \langle\cdot,\cdot\rangle)$ depends on $v$.
    \begin{table}[h]
        \begin{tabular}{|l|l|}\hline
            The value of $v$&The type\\\hline\hline
            ${\displaystyle v\in \bigcup_{n\text{: integer}}\{n\pi\}}$&type I of index 1\\\hline
            ${\displaystyle v\in \bigcup_{n\text{: integer}}(2n\pi, (2n+1)\pi)}$&type II of index 1 ($\varepsilon=-1$)\\\hline
            ${\displaystyle v\in \bigcup_{n\text{: integer}}((2n+1)\pi, (2n+2)\pi)}$&type II of index 1 ($\varepsilon=1$)\\\hline
        \end{tabular}
        \caption{The type of $(A, \langle\cdot,\cdot\rangle)$}\label{tab:3}
    \end{table}
\end{example}
In \cite{MR783023}, we have discussed the existence of a local frame field such that both the shape operator and the metric are forms in Theorem \ref{theo:Petrov} since Theorem \ref{theo:Petrov} is a pointwise claim.
For our latter discussions in Section \ref{sec:4-3}, we now guarantee the existence of such a local frame field.
\begin{proposition}
    \label{prop:1}
    Let $M$ be a non-degenerate hypersurface in the pseudo-Riemannian space form.
    Suppose that $M$ is connected and the Jordan normal form of $A$ is constant on $M$.
    Then there exists a local frame field $\{e_{i}\}$ on an open neighborhood $U_{x}$ of any $x\in M$ such that the matrix representation of $A$ with respect to $\{e_{i}\}$ and the matrix of $\langle\cdot,\cdot\rangle$ relative to $\{e_{i}\}$ correspond with a Petrov normal form on $U_{x}$.
    Moreover, $(A, \langle\cdot,\cdot\rangle)$ has only one algebraic type of Petrov normal forms on $M$.
\end{proposition}
\begin{proof}
    For $x\in M$, we can consider that Theorem \ref{theo:Petrov} applies to a coordinate neighborhood $U_{x}$ around $x$ since the Jordan normal form of $A$ is constant on $M$.
    For all $y\in U_{x}$, the algebraic type of Petrov normal forms of $(A_{y}, \langle\cdot,\cdot\rangle_{y})$ corresponds with the one of $(A_{x}, \langle\cdot,\cdot\rangle_{x})$.
    Then we obtain a local frame field $\{e_{i}\}$ on $U_{x}$ which satisfies both $A$ and $\langle\cdot,\cdot\rangle$ are the forms of its type.

    The latter half of Proposition \ref{prop:1} follows from the connectedness of $M$.
\end{proof}
If $M$ is of index $1$, the minimal polynomials of $A_{x}$ are the following:
\begin{align*}
    \begin{array}{lcllcl}
        \text{I}&\text{:}&(t-a_{0})\cdots (t-a_{n-1}),&\text{II}&\text{:}&(t-b)^{2}(t-a_{1})\cdots (t-a_{n-2}),\\
        \text{III}&\text{:}&(t-b)^{3}(t-a_{1})\cdots (t-a_{n-3}),&\text{IV}&\text{:}&\{(t-\alpha)^{2}+\beta^{2}\}(t-a_{1})\cdots (t-a_{n-2}).
    \end{array}
\end{align*}
Thus, the minimal polynomial of $A$ is constant if and only if the Jordan normal form of $A$ is constant.
Therefore, if the minimal polynomial of $A$ is constant, we can define 4 types I, II, III, and IV of isoparametric hypersurfaces of index 1 by Proposition \ref{prop:1}.
In this paper, we name these types the Petrov types of index 1.

\subsection{A Classification of Isoparametric Hypersurfaces of Index 2}\label{sec:3-3}
Let $M$ be an $n$-dimensional non-degenerate connected hypersurface of index 2 in the pseudo-Riemannian space form.
Let $A$ be the shape operator of $M$.
We refer to \cite[p.\,170]{Upadhyay}.
Similarly to the case of index 1, $(A_{x}, \langle\cdot, \cdot\rangle_{x})$ can be put into one of the following 19 types algebraically by Theorem \ref{theo:Petrov} and Lemma \ref{lemm:1}:
\begin{align*}
    &\text{I} :&
    A_{x}&\sim
    \begin{bmatrix}
        \alpha&-\beta&1&0&&&\\
        \beta&\alpha&0&1&&&\\
        0&0&\alpha&-\beta&&&\\
        0&0&\beta&\alpha&&&\\
        &&&&a_{1}&&\\
        &&&&&\ddots&\\
        &&&&&&a_{n-4}
    \end{bmatrix},
    &&\langle\cdot,\cdot\rangle_{x}\sim
    \begin{bmatrix}
        0&0&-1&0&&&\\
        0&0&0&1&&&\\
        -1&0&0&0&&&\\
        0&1&0&0&&&\\
        &&&&1&&\\
        &&&&&\ddots&\\
        &&&&&&1
    \end{bmatrix},\\
    &\text{II} :&
    A_{x}&\sim
    \begin{bmatrix}
        \alpha_{1}&-\beta_{1}&0&0&&&\\
        \beta_{1}&\alpha_{1}&0&0&&&\\
        0&0&\alpha_{2}&-\beta_{2}&&&\\
        0&0&\beta_{2}&\alpha_{2}&&&\\
        &&&&a_{1}&&\\
        &&&&&\ddots&\\
        &&&&&&a_{n-4}
    \end{bmatrix},
    &&\langle\cdot,\cdot\rangle_{x}\sim
    \begin{bmatrix}
        -1&0&0&0&&&\\
        0&1&0&0&&&\\
        0&0&-1&0&&&\\
        0&0&0&1&&&\\
        &&&&1&&\\
        &&&&&\ddots&\\
        &&&&&&1
    \end{bmatrix},\\
    &\text{III} :&
    A_{x}&\sim
    \begin{bmatrix}
        \alpha&-\beta&&&&\\
        \beta&\alpha&&&&\\
        &&a_{0}&&\\
        &&&a_{1}&&\\
        &&&&\ddots&\\
        &&&&&a_{n-3}
    \end{bmatrix},
    &&\langle\cdot,\cdot\rangle_{x}\sim
    \begin{bmatrix}
        -1&0&&&&\\
        0&1&&&&\\
        &&-1&&\\
        &&&1&&\\
        &&&&\ddots&\\
        &&&&&1
    \end{bmatrix},\\
    &\text{IV} :&
    A_{x}&\sim
    \begin{bmatrix}
        \alpha&-\beta&&&&&\\
        \beta&\alpha&&&&&\\
        &&b&1&&\\
        &&0&b&&&\\
        &&&&a_{1}&&\\
        &&&&&\ddots&\\
        &&&&&&a_{n-4}
    \end{bmatrix},
    &&\langle\cdot,\cdot\rangle_{x}\sim
    \begin{bmatrix}
        -1&0&&&&&\\
        0&1&&&&&\\
        &&0&\varepsilon&&\\
        &&\varepsilon&0&&&\\
        &&&&1&&\\
        &&&&&\ddots&\\
        &&&&&&1
    \end{bmatrix},\\
    &\text{V} :&
    A_{x}&\sim
    \begin{bmatrix}
        \alpha&-\beta&&&&&&\\
        \beta&\alpha&&&&&&\\
        &&b&1&0&&\\
        &&0&b&1&&&\\
        &&0&0&b&&&\\
        &&&&&a_{1}&&\\
        &&&&&&\ddots&\\
        &&&&&&&a_{n-5}
    \end{bmatrix},
    &&\langle\cdot,\cdot\rangle_{x}\sim
    \begin{bmatrix}
        -1&0&&&&&&\\
        0&1&&&&&&\\
        &&0&0&1&&\\
        &&0&1&0&&&\\
        &&1&0&0&&&\\
        &&&&&1&&\\
        &&&&&&\ddots&\\
        &&&&&&&1
    \end{bmatrix},\\
    &\text{VI} :&
    A_{x}&\sim
    \begin{bmatrix}
        b&1&0&0&&&\\
        0&b&1&0&&&\\
        0&0&b&1&&&\\
        0&0&0&b&&&\\
        &&&&a_{1}&&\\
        &&&&&\ddots&\\
        &&&&&&a_{n-4}
    \end{bmatrix},
    &&\langle\cdot,\cdot\rangle_{x}\sim
    \begin{bmatrix}
        0&0&0&\varepsilon&&&\\
        0&0&\varepsilon&0&&&\\
        0&\varepsilon&0&0&&&\\
        \varepsilon&0&0&0&&&\\
        &&&&1&&\\
        &&&&&\ddots&\\
        &&&&&&1
    \end{bmatrix},\\
    &\text{VII-i} :&
    A_{x}&\sim
    \begin{bmatrix}
        b&1&0&&&\\
        0&b&1&&&\\
        0&0&b&&&\\
        &&&a_{1}&&\\
        &&&&\ddots&\\
        &&&&&a_{n-3}
    \end{bmatrix},
    &&\langle\cdot,\cdot\rangle_{x}\sim
    \begin{bmatrix}
        0&0&-1&&&\\
        0&-1&0&&&\\
        -1&0&0&&&\\
        &&&1&&\\
        &&&&\ddots&\\
        &&&&&1
    \end{bmatrix},\\
    &\text{VII-ii} :&
    A_{x}&\sim
    \begin{bmatrix}
        b&1&0&&&&\\
        0&b&1&&&&\\
        0&0&b&&&&\\
        &&&a_{0}&&&\\
        &&&&a_{1}&&\\
        &&&&&\ddots&\\
        &&&&&&a_{n-4}
    \end{bmatrix},
    &&\langle\cdot,\cdot\rangle_{x}\sim
    \begin{bmatrix}
        0&0&1&&&&\\
        0&1&0&&&&\\
        1&0&0&&&&\\
        &&&-1&&&\\
        &&&&1&&\\
        &&&&&\ddots&\\
        &&&&&&1
    \end{bmatrix},\\
    &\text{VIII} :&
    A_{x}&\sim
    \begin{bmatrix}
        b_{1}&1&0&&&&&\\
        0&b_{1}&1&&&&&\\
        0&0&b_{1}&&&&&\\
        &&&b_{2}&1&&&\\
        &&&0&b_{2}&&&\\
        &&&&&a_{1}&&\\
        &&&&&&\ddots&\\
        &&&&&&&a_{n-5}
    \end{bmatrix},
    &&\langle\cdot,\cdot\rangle_{x}\sim
    \begin{bmatrix}
        0&0&1&&&&&\\
        0&1&0&&&&&\\
        1&0&0&&&&&\\
        &&&0&\varepsilon&&&\\
        &&&\varepsilon&0&&&\\
        &&&&&1&&\\
        &&&&&&\ddots&\\
        &&&&&&&1
    \end{bmatrix},\\
    &\text{IX-i} :&
    A_{x}&\sim
    \begin{bmatrix}
        b_{1}&1&&&&&\\
        0&b_{1}&&&&&\\
        &&b_{2}&1&&&\\
        &&0&b_{2}&&&\\
        &&&&a_{1}&&\\
        &&&&&\ddots&\\
        &&&&&&a_{n-4}
    \end{bmatrix},
    &&\langle\cdot,\cdot\rangle_{x}\sim
    \begin{bmatrix}
        0&\varepsilon&&&&&\\
        \varepsilon&0&&&&&\\
        &&0&\varepsilon&&&\\
        &&\varepsilon&0&&&\\
        &&&&1&&\\
        &&&&&\ddots&\\
        &&&&&&1
    \end{bmatrix},\\
    &\text{IX-ii} :&
    A_{x}&\sim
    \begin{bmatrix}
        b_{1}&1&&&&&\\
        0&b_{1}&&&&&\\
        &&b_{2}&1&&&\\
        &&0&b_{2}&&&\\
        &&&&a_{1}&&\\
        &&&&&\ddots&\\
        &&&&&&a_{n-4}
    \end{bmatrix},
    &&\langle\cdot,\cdot\rangle_{x}\sim
    \begin{bmatrix}
        0&\varepsilon&&&&&\\
        \varepsilon&0&&&&&\\
        &&0&-\varepsilon&&&\\
        &&-\varepsilon&0&&&\\
        &&&&1&&\\
        &&&&&\ddots&\\
        &&&&&&1
    \end{bmatrix},\\
    &\text{X} :&
    A_{x}&\sim
    \begin{bmatrix}
        b&1&&&&\\
        0&b&&&&\\
        &&a_{0}&&&\\
        &&&a_{1}&&\\
        &&&&\ddots&\\
        &&&&&a_{n-3}
    \end{bmatrix},
    &&\langle\cdot,\cdot\rangle_{x}\sim
    \begin{bmatrix}
        0&\varepsilon&&&&\\
        \varepsilon&0&&&&\\
        &&-1&&&\\
        &&&1&&\\
        &&&&\ddots&\\
        &&&&&1
    \end{bmatrix},\\
    &\text{XI} :&
    A_{x}&\sim
    \begin{bmatrix}
        a_{-1}&&&&\\
        &a_{0}&&&\\
        &&a_{1}&&\\
        &&&\ddots&\\
        &&&&a_{n-2}
    \end{bmatrix},
    &&\langle\cdot,\cdot\rangle_{x}\sim
    \begin{bmatrix}
        -1&&&&\\
        &-1&&&\\
        &&1&&\\
        &&&\ddots&\\
        &&&&1
    \end{bmatrix},
\end{align*}
where $a_{1}, a_{2}, \ldots, a_{n-2}, a_{0}, a_{-1}, b, b_{1}, b_{2}, \alpha, \beta, \alpha_{1}, \beta_{1}, \alpha_{2}, \beta_{2}\in \mathbb{R}$, and where $\beta, \beta_{1}, \beta_{2}>0$, and where $\varepsilon\in \{-1, 1\}$.
We note that we distinguish between types where $\varepsilon=1$ and types where $\varepsilon=-1$ algebraically.
Similarly to the case of index 1, the above forms are different from the forms in Theorem \ref{theo:Petrov}.
In types IV, VI, VIII, IX-i, IX-ii, and X of index 2, we can select the signs.
Moreover, we can modify the bases to reverse the signs.
Let the left bases in \eqref{eq:bases2} be bases of the above forms.
\begin{equation}
    \begin{array}{lclcl}
        \text{IV}&\text{:}&C_{1}, C_{2}, X_{1}, X_{2}, Y_{1}, \ldots, Y_{n-4}&\xrightarrow{\xi\to-\xi}&-C_{1}, C_{2}, -X_{1}, X_{2}, Y_{1}, \ldots, Y_{n-4},\\
        \text{VI}&\text{:}&X_{1}, X_{2}, X_{3}, X_{4}, Y_{1}, \ldots, Y_{n-4}&\xrightarrow{\xi\to-\xi}&-X_{1}, X_{2}, -X_{3}, X_{4}, Y_{1}, \ldots, Y_{n-4},\\
        \text{VIII}&\text{:}&X_{1}, X_{2}, X_{3}, Y_{1}, Y_{2}, Z_{1}, \ldots, Z_{n-5}&\xrightarrow{\xi\to-\xi}&-X_{1}, X_{2}, -X_{3}, -Y_{1}, Y_{2}, Z_{1}, \ldots, Z_{n-5},\\
        \text{IX}&\text{:}&X_{1}, X_{2}, Y_{1}, Y_{2}, Z_{1}, \ldots, Z_{n-4}&\xrightarrow{\xi\to-\xi}&-X_{1}, X_{2}, -Y_{1}, Y_{2}, Z_{1}, \ldots, Z_{n-4},\\
        \text{X}&\text{:}&X_{1}, X_{2}, Y, Z_{1}, \ldots, Z_{n-3} &\xrightarrow{\xi\to-\xi}&-X_{1}, X_{2}, Y, Z_{1}, \ldots, Z_{n-3}.
    \end{array}
    \label{eq:bases2}
\end{equation}
The right bases in \eqref{eq:bases2} are bases of the above forms where the sign of $\varepsilon$ is the opposite.
By considering the direction of the unit normal vectors at $x$, we can reduce the above 19 types to 13 types geometrically.
Since Theorem \ref{theo:Petrov} is a pointwise claim, $(A, \langle\cdot,\cdot\rangle)$ can have multiple types on $M$.
One of the examples is Example \ref{exam:0-2}.
\begin{example}\label{exam:0-2}
    There exists an isoparametric hypersurface whose $(A, \langle\cdot,\cdot\rangle)$ has 3 types and the minimal polynomial of $A$ is constant on it.
    Let
    \begin{equation*}
        \bm{a}_{1}
        :=
        \frac{\bm{e}_{1}+\bm{e}_{3}}{\sqrt{2}}
        ,
        \quad
        \bm{a}_{2}
        :=
        \frac{\bm{e}_{1}-\bm{e}_{3}}{\sqrt{2}}
        ,
        \quad
        \bm{a}_{3}
        :=
        \frac{\bm{e}_{2}+\bm{e}_{4}}{\sqrt{2}}
        ,
        \quad
        \bm{a}_{4}
        :=
        \frac{\bm{e}_{2}-\bm{e}_{4}}{\sqrt{2}}
    \end{equation*}
    where $\bm{e}_{1}, \bm{e}_{2}, \bm{e}_{3}, \bm{e}_{4}\in \mathbb{R}^{5}$.
    Let $f : \mathbb{R}^{4}\to \mathbb{R}^{5}_{2}$ be the mapping
    \begin{equation*}
        f(\bm{p}):=x\bm{a}_{1}+y\bm{a}_{2}+z\bm{a}_{3}+w\bm{a}_{4}+\left(\frac{y^{2}}{2}-\sin{w}\right)\bm{e}_{5},\quad \bm{p}=(x, y, z, w)\in \mathbb{R}^{4},
    \end{equation*}
    where $\bm{e}_{5}\in \mathbb{R}^{5}$.
    Let $M:=f(\mathbb{R}^{4})$.
    The shape operator $A_{f(\bm{p})}$ at $\bm{p}\in \mathbb{R}^{4}$ satisfies
    \begin{equation*}
        \xi_{f(\bm{p})}
        =
        \begin{pmatrix}
            {\displaystyle -\frac{y}{\sqrt{2}}}\vspace{1mm}\\
            {\displaystyle \frac{\cos{w}}{\sqrt{2}}}\vspace{1mm}\\
            {\displaystyle -\frac{y}{\sqrt{2}}}\vspace{1mm}\\
            {\displaystyle \frac{\cos{w}}{\sqrt{2}}}\vspace{1mm}\\
            -1
        \end{pmatrix}
        ,\quad
        \mathcal{B}_{\bm{p}}
        :=
        \begin{bmatrix}
            {\displaystyle\frac{\partial f}{\partial x}(\bm{p})}&{\displaystyle\frac{\partial f}{\partial y}(\bm{p})}&{\displaystyle\frac{\partial f}{\partial z}(\bm{p})}&{\displaystyle\frac{\partial f}{\partial w}(\bm{p})}
        \end{bmatrix}
        ,\quad
        A_{f(\bm{p})}\mathcal{B}_{\bm{p}}
        =
        \mathcal{B}_{\bm{p}}
        \begin{pmatrix}
            0&1&0&0\\
            0&0&0&0\\
            0&0&0&\sin{w}\\
            0&0&0&0
        \end{pmatrix}
        .
    \end{equation*}
    Thus $M$ is an isoparametric hypersurface in $\mathbb{R}^{5}_{2}$.
    By Table \ref{tab:2}, the type of $(A, \langle\cdot,\cdot\rangle)$ depends on $w$.
    \begin{table}[h]
        \begin{tabular}{|l|l|}\hline
            The value of $w$&The type\\\hline\hline
            ${\displaystyle w\in \bigcup_{n\text{: integer}}\{n\pi\}}$&type X of index 2\\\hline
            ${\displaystyle w\in \bigcup_{n\text{: integer}}(2n\pi, (2n+1)\pi)}$&type IX-i of index 2\\\hline
            ${\displaystyle w\in \bigcup_{n\text{: integer}}((2n+1)\pi, (2n+2)\pi)}$&type IX-ii of index 2\\\hline
        \end{tabular}
        \caption{The type of $(A, \langle\cdot,\cdot\rangle)$}\label{tab:2}
    \end{table}
    On the other hand, the minimal polynomial of $A$ is $t^{2}$ on $M$, which does not depend on $w$.
\end{example}
Therefore, if the Jordan normal form of $A$ is constant, we can define 13 types I, II, ..., and XI of isoparametric hypersurfaces of index 2 by Proposition \ref{prop:1}.
In this paper, we name these types the Petrov types of index 2.

\section{4-Dimensional Isoparametric Hypersurfaces of Index 2 in the Pseudo-Riemannian Space Forms}\label{sec:4}
In this section, we investigate 4-dimensional isoparametric hypersurfaces of index 2 in the pseudo-Riemannian space forms (isoparametric hypersurfaces of index 2 in the 5-dimensional pseudo-Riemannian space forms).

Let $M^{4}_{2}$ be a 4-dimensional hypersurface of index $2$ in the pseudo-Riemannian space form.
The ambient spaces of $M$ can be $\mathbb{R}^{5}_{2}, \mathbb{R}^{5}_{3}, S^{5}_{2}, S^{5}_{3}, H^{5}_{2}$, or $H^{5}_{3}$.
By considering the anti-isometry, we only have to consider $\mathbb{R}^{5}_{2}, S^{5}_{2}$, or $S^{5}_{3}$ as the ambient space of $M^{4}_{2}$.
Then the types of $M$ are $\delta=0$, $1$, or $-1$, respectively.
Since the dimension of $M$ is $4$, the Petrov type of $M$ can not be type V or VIII of index 2.
Therefore, the number of the cases of $M$ is $3\times 11=33$.

From now on, we investigate the existence and non-existence in each case of $M$ one by one.
For example, we have obtained the examples of types IX-i and IX-ii of index 2 by Example \ref{exam:0-2}.
Here, we arrange the cases of $M$ in Table \ref{tab:1}.
\begin{table}[htbp]
    \center
    \begin{tabular}{c||c|c|c|c||c|c|c|c|c|c|c|}
        &I&II&III&IV&VI&VII-i&VII-ii&IX-i&IX-ii&X&XI\\\hline\hline
        $\mathbb{R}^{5}_{2}$ ($\delta=0$)&$\times$&$\triangle$&$\times$&$\times$&\ref{exam:m}&\ref{exam:l}&\ref{exam:k}&\ref{exam:d}&\ref{exam:c}&\ref{exam:b}&\ref{exam:a}\\\hline
        $S^{5}_{2}$ ($\delta=1$)&$\times$&$\triangle$&&$\times$&&&&&&&\ref{exam:e}\\\hline
        $S^{5}_{3}$ ($\delta=-1)$&&\ref{exam:j}&&&&&&\ref{exam:i}&\ref{exam:h}&\ref{exam:g}&\ref{exam:f}\\\hline
    \end{tabular}
    \caption{the cases of $M$}\label{tab:1}
\end{table}
In VI, VII-i, ..., or XI, the principal curvatures of $M$ are only real numbers.
In the case where one letter is written in Table \ref{tab:1}, there exists such an isoparametric hypersurface, which is explained in Sections \ref{sec:4-1} and \ref{sec:4-2}.
The symbols $\times$ and $\triangle$ in Table \ref{tab:1} mean non-existence and conditional non-existence, respectively, which are explained in Section \ref{sec:4-3}.
We remain open questions at the blanks and the triangle symbols in Table \ref{tab:1}.
We anticipate the claims that there exist no isoparametric hypersurfaces whose shape operators have complex principal curvatures in $\delta=0$ or $\delta=1$ and that there exist isoparametric hypersurfaces whose shape operators have only real principal curvatures in all cases.

\subsection{Hahn's Examples}\label{sec:4-1}
We construct several examples by using \cite[2.3, 2.4]{MR753432}.

\begin{example}[{\cite[2.3]{MR753432}}]
    Let $f\colon \mathbb{R}^{n}_{s}\to \mathbb{R}$ be the function $f(x):=\langle Px, x\rangle_{s}+2\langle p, x\rangle_{s}$ where $0\neq P\in \mathrm{Sym}(\mathbb{R}^{n}_{s}, \langle\cdot,\cdot\rangle_{s})$ and $p\in \mathbb{R}^{n}_{s}$.
    If there exists $\rho\in \mathbb{R}\setminus \{0\}$ such that $P=\rho E_{n}$, or if $P^{2}=O$, $Pp=\bm{0}$, and $\langle p, p\rangle\neq 0$, then $f$ is isoparametric and $\mathrm{W}_{\mathrm{RN}}(f)\neq \emptyset$.
    Thus, for $c\in \mathrm{W}_{\mathrm{RN}}(f)$, a connected component $M$ of $f^{-1}(c)\subset \mathbb{R}^{n}_{s}$ is an isoparametric hypersurface.
    For $x\in M$, we can compute $\mathrm{grad}^{\mathbb{R}^{n}_{s}}(f)_{x}=2Px+2p$.
    \begin{enumerate}[label=(\alph*)]
        \item \label{exam:a}
        Let $n=5$, $s=2$, $P=-E_{5}$, $p=\bm{0}$, and $f(\bm{e}_{3})=c=-1$.
        Then $A=I$ and $f^{-1}(c)=M=S^{4}_{2}$ hold.
        \item \label{exam:b}
        Let $n=5$, $s=2$, $p=\bm{e}_{3}\in \mathbb{R}^{5}$, $f(\bm{e}_{1})=c=1$, and
        \begin{equation*}
            P
            =
            \begin{pmatrix}
                -1&0&0&0&1\\
                0&0&0&0&0\\
                0&0&0&0&0\\
                0&0&0&0&0\\
                -1&0&0&0&1
            \end{pmatrix}
            .
        \end{equation*}
        The shape operator $A_{x}$ at any $x=(x^{1}, \ldots, x^{5})\in M\subset \mathbb{R}^{5}_{2}$ satisfies
        \begin{align*}
            \mathcal{B}_{x}
            &:=
            \begin{bmatrix}
                \bm{e}_{1}+\bm{e}_{5}&\bm{e}_{1}+(-x^{1}+x^{5})\bm{e}_{3}&\bm{e}_{2}&\bm{e}_{4}
            \end{bmatrix},
            &
            A_{x}
            \mathcal{B}_{x}
            &=
            \mathcal{B}_{x}
            \left(J_{2}(0)\oplus O_{2}\right)
            .
        \end{align*}
        \item \label{exam:c}
        Let $n=5$, $s=2$, $p=\bm{e}_{4}\in \mathbb{R}^{5}$, $f(\bm{e}_{4}/2)=c=1$, and
        \begin{equation*}
            P
            =
            \begin{pmatrix}
                0&1&0&0&1\\
                1&0&-1&0&0\\
                0&1&0&0&1\\
                0&0&0&0&0\\
                -1&0&1&0&0
            \end{pmatrix}
            .
        \end{equation*}
        The shape operator $A_{x}$ at any $x=(x^{1}, \ldots, x^{5})\in M\subset \mathbb{R}^{5}_{2}$ satisfies
        \begin{align*}
            \begin{bmatrix}
                \theta_{1}&\theta_{2}(x)&\eta_{1}&\eta_{2}(x)
            \end{bmatrix}
            &:=
            \begin{bmatrix}
                -\bm{e}_{1}-\bm{e}_{3}&\bm{e}_{2}+(x^{1}-x^{3})\bm{e}_{4}&-\bm{e}_{2}+\bm{e}_{5}&\bm{e}_{1}+(x^{2}+x^{5})\bm{e}_{4}
            \end{bmatrix},\\
            \begin{bmatrix}
                \theta_{1}'&\theta_{2}'(x)&\eta_{1}'&\eta_{2}'(x)
            \end{bmatrix}
            &:=
            \begin{bmatrix}
                \theta_{1}+\eta_{1}&\theta_{2}(x)+\eta_{2}(x)&\theta_{1}-\eta_{1}&\theta_{2}(x)-\eta_{2}(x)
            \end{bmatrix},\\
            S(x)&:=x^{1}+x^{2}-x^{3}+x^{5},\quad T(x):=x^{1}-x^{2}-x^{3}-x^{5},\\
            \begin{bmatrix}
                \theta_{1}''&\theta_{2}''(x)&\eta_{1}''&\eta_{2}''(x)
            \end{bmatrix}
            &:=
            \begin{bmatrix}
                \theta_{1}'&\theta_{2}'(x)+(1/2)S(x)T(x)\eta_{1}'&\eta_{1}'&\eta_{2}'(x)
            \end{bmatrix},
        \end{align*}
        \begin{align*}
            \bm{b}_{1}&:=(1/\sqrt{2})\theta_{1}''
            ,\quad
            \bm{b}_{2}(x):=(1/\sqrt{2})\theta_{2}''(x)-(1/2\sqrt{2})\left\{(S(x))^{2}-2\right\}\theta_{1}'',\\
            \bm{b}_{3}&:=(1/\sqrt{2})\eta_{1}''
            ,\quad
            \bm{b}_{4}(x):=(1/\sqrt{2})\eta_{2}''(x)-(1/2\sqrt{2})\left\{(T(x))^{2}-2\right\}\eta_{1}'',\\
            \mathcal{B}_{x}
            &:=
            \begin{bmatrix}
                \bm{b}_{1}&\bm{b}_{2}(x)&\bm{b}_{3}&\bm{b}_{4}(x)
            \end{bmatrix}
            ,\\
            A_{x}
            \mathcal{B}_{x}
            &=
            \mathcal{B}_{x}
            \left(J_{2}(0)\oplus J_{2}(0)\right)
            ,\quad
            (\langle \bm{b}_{i}, \bm{b}_{j}\rangle_{2})
            =
            E^{\mathfrak{a}}_{2}\oplus (-E^{\mathfrak{a}}_{2})
            .
        \end{align*}
        \item \label{exam:d}
        Let $n=5$, $s=2$, $p=\bm{e}_{4}\in \mathbb{R}^{5}$, $f(\bm{e}_{4}/2)=c=1$, and
        \begin{equation*}
            P
            =
            \begin{pmatrix}
                1&0&0&0&1\\
                0&1&-1&0&0\\
                0&1&-1&0&0\\
                0&0&0&0&0\\
                -1&0&0&0&-1
            \end{pmatrix}
            .
        \end{equation*}
        The shape operator $A_{x}$ at any $x=(x^{1}, \ldots, x^{5})\in M\subset \mathbb{R}^{5}_{2}$ satisfies
        \begin{align*}
            \begin{bmatrix}
                \theta_{1}&\theta_{2}(x)&\eta_{1}&\eta_{2}(x)
            \end{bmatrix}
            &:=
            \begin{bmatrix}
                -\bm{e}_{1}+\bm{e}_{5}&\bm{e}_{1}+(x^{1}+x^{5})\bm{e}_{4}&-\bm{e}_{2}-\bm{e}_{3}&\bm{e}_{2}+(x^{2}-x^{3})\bm{e}_{4}
            \end{bmatrix},\\
            \begin{bmatrix}
                \theta_{1}'&\theta_{2}'(x)&\eta_{1}'&\eta_{2}'(x)
            \end{bmatrix}
            &:=
            \begin{bmatrix}
                \theta_{1}&\theta_{2}-(x^{1}+x^{5})(x^{2}-x^{3})\eta_{1}&\eta_{1}&\eta_{2}
            \end{bmatrix},
        \end{align*}
        \begin{align*}
            \bm{b}_{1}&:=\theta_{1}'
            ,\quad
            \bm{b}_{2}(x):=\theta_{2}'(x)-(1/2)\left\{(x^{1}+x^{5})^{2}-1\right\}\theta_{1}',\\
            \bm{b}_{3}&:=\eta_{1}'
            ,\quad
            \bm{b}_{4}(x):=\eta_{2}'(x)-(1/2)\left\{(x^{2}-x^{3})^{2}-1\right\}\eta_{1}',\\
            \mathcal{B}_{x}&:=
            \begin{bmatrix}
                \bm{b}_{1}&\bm{b}_{2}(x)&\bm{b}_{3}&\bm{b}_{4}(x)
            \end{bmatrix}
            ,\\
            A_{x}
            \mathcal{B}_{x}
            &=
            \mathcal{B}_{x}
            \left(J_{2}(0)\oplus J_{2}(0)\right)
            ,\quad
            (\langle \bm{b}_{i}, \bm{b}_{j}\rangle_{2})
            =
            E^{\mathfrak{a}}_{2}\oplus E^{\mathfrak{a}}_{2}
            .
        \end{align*}
    \end{enumerate}
\end{example}

\begin{example}[{\cite[2.4]{MR753432}}]
    Let $f\colon S^{n}_{s}\to \mathbb{R}$ be the function $f(x):=\langle Px, x\rangle_{s}$ where $P\in \mathrm{Sym}(\mathbb{R}^{n+1}_{s}, \langle\cdot,\cdot\rangle_{s})$.
    If the degree of the minimal polynomial $\mu_{P}$ of $P$ is 2, then $f$ is isoparametric and $\mathrm{W}_{\mathrm{RN}}(f)\neq \emptyset$.
    Thus, for $c\in \mathrm{W}_{\mathrm{RN}}(f)$, a connected component $M$ of $f^{-1}(c)$ is an isoparametric hypersurface.
    For $x\in M$, we can compute $\mathrm{grad}^{S^{n}_{s}}(f)_{x}=2Px-2\langle Px, x\rangle_{s} x$.
    Then the shape operator of $M$ at $x\in M$ is
    \begin{equation*}
        A_{x}
        =
        \frac{1}{\sqrt{-\delta \mu_{P}(c)}}\left.\left(cE_{n+1}-P\right)\right|_{T_{x}M}
        ,
    \end{equation*}
    where $\delta=\nu=\mathrm{sign}\left(\langle\mathrm{grad}^{S^{n}_{s}}(f)_{x}, \mathrm{grad}^{S^{n}_{s}}(f)_{x}\rangle\right)$ is the type of $M$.
    \begin{enumerate}[label=(\alph*)]
        \setcounter{enumi}{4}
        \item \label{exam:e}
        Let $n=5$, $s=2$, $f(\bm{e}_{6})=c=0$, and $P=E^{\mathfrak{a}}_{2}\oplus E^{\mathfrak{a}}_{4}$.
        Then $\mu_{P}(t)=t^2-1$.
        The shape operator $A$ is diagonalizable.
        The principal curvatures of $M$ are $-1$, $-1$, $1$, and $1$.
        \item \label{exam:f}
        Let $n=5$, $s=3$, $f(-\bm{e}_{1}/\sqrt{2}+\bm{e}_{3}/\sqrt{2}+\sqrt{2}\bm{e}_{5})=c=3$, and $P=E^{\mathfrak{a}}_{3}\oplus E^{\mathfrak{a}}_{3}$.
        Then $\mu_{P}(t)=t^2-1$.
        The shape operator $A$ is diagonalizable.
        The principal curvatures of $M$ are $1/\sqrt{2}$, $1/\sqrt{2}$, $1/\sqrt{2}$, and $\sqrt{2}$.
        \item \label{exam:g}
        Let $n=5$, $s=3$, $f(\bm{e}_{4}/\sqrt{2}+\bm{e}_{6}/\sqrt{2})=c=1$, and
        \begin{equation*}
            P
            =
            \begin{pmatrix}
                0&0&1&0&0&-1\\
                0&0&0&0&0&0\\
                1&0&0&1&0&0\\
                0&0&-1&0&0&1\\
                0&0&0&0&0&0\\
                1&0&0&1&0&0
            \end{pmatrix}
            .
        \end{equation*}
        Then $\mu_{P}(t)=t^2$.
        For $x=(x^{1}, \ldots, x^{6})\in M\subset S^{5}_{3}\subset \mathbb{R}^{6}_{3}$ where $x^1+x^4\neq 0$ and $-x^{3}+x^{6}\neq 0$, we define
        \begin{align*}
            \bm{a}_{1}(x)
            &:=
            -\frac{x^{2}}{x^{1}+x^{4}}\bm{e}_{1}+\bm{e}_{2}+\frac{x^{2}}{x^{1}+x^{4}}\bm{e}_{4},\\
            \bm{a}_{2}(x)
            &:=
            \left(-\frac{x^{3}}{x^{1}+x^{4}}+\frac{x^{4}}{-x^{3}+x^{6}}\right)\bm{e}_{1}+\bm{e}_{3}+\left(\frac{x^{3}}{x^{1}+x^{4}}+\frac{x^{1}}{-x^{3}+x^{6}}\right)\bm{e}_{4},\\
            \bm{a}_{3}(x)
            &:=
            \frac{x^{5}}{x^{1}+x^{4}}\bm{e}_{1}-\frac{x^{5}}{x^{1}+x^{4}}\bm{e}_{4}+\bm{e}_{5},\\
            \bm{a}_{4}(x)
            &:=
            \left(\frac{x^{6}}{x^{1}+x^{4}}-\frac{x^{4}}{-x^{3}+x^{6}}\right)\bm{e}_{1}+\left(-\frac{x^{6}}{x^{1}+x^{4}}-\frac{x^{1}}{-x^{3}+x^{6}}\right)\bm{e}_{4}+\bm{e}_{6}.
        \end{align*}
        Then $\bm{a}_{i}(x)\in T_{x}M\ (i\in \{1, 2, 3, 4\})$ holds.
        The shape operator $A_{x}$ at $x\in M$ satisfies
        \begin{align*}
            \mathcal{B}_{x}
            &:=
            \begin{bmatrix}
                \bm{a}_{2}(x)+\bm{a}_{4}(x)&{\displaystyle \left(-\frac{-x^{3}+x^{6}}{x^{1}+x^{4}}\right)\bm{a}_{2}(x)}&\bm{a}_{1}(x)&\bm{a}_{3}(x)
            \end{bmatrix},
            &
            A_{x}
            \mathcal{B}_{x}
            &=
            \mathcal{B}_{x}
            \left(J_{2}(1)\oplus E_{2}\right)
            .
        \end{align*}
        \item \label{exam:h}
        Let $n=5$, $s=3$, $f(\bm{e}_{5})=c=-1$, and
        \begin{equation*}
            P
            =
            \begin{bmatrix}
                E^{\mathfrak{a}}_{3}&E_{3}\\
                -E_{3}&-E^{\mathfrak{a}}_{3}
            \end{bmatrix}
            .
        \end{equation*}
        Then $\mu_{P}(t)=t^2$.
        For $x=(x^{1}, \ldots, x^{6})\in M\subset S^{5}_{3}\subset \mathbb{R}^{6}_{3}$ where $x^2+x^5\neq 0$, we define
        \begin{align*}
            \bm{a}_{1}(x)
            &:=
            \bm{e}_{1}+\left(-\frac{x^{1}}{x^{2}+x^{5}}-x^{5}\frac{x^{3}+x^{4}}{(x^{2}+x^{5})^{2}}\right)\bm{e}_{2}+\left(\frac{x^{1}}{x^{2}+x^{5}}-x^{2}\frac{x^{3}+x^{4}}{(x^{2}+x^{5})^{2}}\right)\bm{e}_{5},\\
            \bm{a}_{2}(x)
            &:=
            \left(-\frac{x^{3}}{x^{2}+x^{5}}-x^{5}\frac{x^{1}+x^{6}}{(x^{2}+x^{5})^{2}}\right)\bm{e}_{2}+\bm{e}_{3}+\left(\frac{x^{3}}{x^{2}+x^{5}}-x^{2}\frac{x^{1}+x^{6}}{(x^{2}+x^{5})^{2}}\right)\bm{e}_{5},\\
            \bm{a}_{3}(x)
            &:=
            \left(\frac{x^{4}}{x^{2}+x^{5}}-x^{5}\frac{x^{1}+x^{6}}{(x^{2}+x^{5})^{2}}\right)\bm{e}_{2}+\bm{e}_{4}+\left(-\frac{x^{4}}{x^{2}+x^{5}}-x^{2}\frac{x^{1}+x^{6}}{(x^{2}+x^{5})^{2}}\right)\bm{e}_{5},\\
            \bm{a}_{4}(x)
            &:=
            \left(\frac{x^{6}}{x^{2}+x^{5}}-x^{5}\frac{x^{3}+x^{4}}{(x^{2}+x^{5})^{2}}\right)\bm{e}_{2}+\left(-\frac{x^{6}}{x^{2}+x^{5}}-x^{2}\frac{x^{3}+x^{4}}{(x^{2}+x^{5})^{2}}\right)\bm{e}_{5}+\bm{e}_{6},
        \end{align*}
        Then $\bm{a}_{i}(x)\in T_{x}M\ (i\in \{1, 2, 3, 4\})$ holds.
        The shape operator $A_{x}$ at $x\in M$ satisfies
        \begin{align*}
            \mathcal{B}_{x}
            &:=
            \begin{bmatrix}
                -\bm{a}_{1}(x)+\bm{a}_{4}(x)&\bm{a}_{2}(x)&-\bm{a}_{2}(x)+\bm{a}_{3}(x)&\bm{a}_{1}(x)
            \end{bmatrix},
            &
            A_{x}
            \mathcal{B}_{x}
            &=
            \mathcal{B}_{x}
            \left(J_{2}(-1)\oplus J_{2}(-1)\right)
            .
        \end{align*}
        Moreover, the shape operator $A_{\bm{e}_{5}}$ satisfies
        \begin{align*}
            \bm{b}_{1}
            &:=
            \frac{1}{\sqrt{2}}\left(-\bm{a}_{1}(\bm{e}_{5})-\bm{a}_{2}(\bm{e}_{5})+\bm{a}_{3}(\bm{e}_{5})+\bm{a}_{4}(\bm{e}_{5})\right),\\
            \bm{b}_{2}
            &:=
            \frac{1}{2\sqrt{2}}\left(\bm{a}_{1}(\bm{e}_{5})+\bm{a}_{2}(\bm{e}_{5})+\bm{a}_{3}(\bm{e}_{5})+\bm{a}_{4}(\bm{e}_{5})\right),\\
            \bm{b}_{3}
            &:=
            \frac{1}{\sqrt{2}}\left(-\bm{a}_{1}(\bm{e}_{5})+\bm{a}_{2}(\bm{e}_{5})-\bm{a}_{3}(\bm{e}_{5})+\bm{a}_{4}(\bm{e}_{5})\right),\\
            \bm{b}_{4}
            &:=
            \frac{1}{2\sqrt{2}}\left(-\bm{a}_{1}(\bm{e}_{5})+\bm{a}_{2}(\bm{e}_{5})+\bm{a}_{3}(\bm{e}_{5})-\bm{a}_{4}(\bm{e}_{5})\right),\\
            \mathfrak{B}_{\bm{e}_{5}}
            &:=
            \begin{bmatrix}
                \bm{b}_{1} & \bm{b}_{2} & \bm{b}_{3} & \bm{b}_{4}
            \end{bmatrix}
            ,\\
            A_{\bm{e}_{5}}
            \mathfrak{B}_{\bm{e}_{5}}
            &=
            \mathfrak{B}_{\bm{e}_{5}}
            \left(J_{2}(-1)\oplus J_{2}(-1)\right)
            ,\quad
            (\langle \bm{b}_{i}, \bm{b}_{j}\rangle_{3})
            =
            E^{\mathfrak{a}}_{2}\oplus (-E^{\mathfrak{a}}_{2})
            .
        \end{align*}
        \item \label{exam:i}
        Let $n=5$, $s=3$, $f(\bm{e}_{5})=c=-1$, and
        \begin{equation*}
            P
            =
            \begin{bmatrix}
                E_{3}&E_{3}\\
                -E_{3}&-E_{3}
            \end{bmatrix}
            .
        \end{equation*}
        Then $\mu_{P}(t)=t^2$.
        For $x=(x^{1}, \ldots, x^{6})\in M\subset S^{5}_{3}\subset \mathbb{R}^{6}_{3}$ where $x^2+x^5\neq 0$, we define
        \begin{align*}
            \bm{a}_{1}(x)
            &:=
            \bm{e}_{1}+\left(-\frac{x^{1}}{x^{2}+x^{5}}-x^{5}\frac{x^{1}+x^{4}}{(x^{2}+x^{5})^{2}}\right)\bm{e}_{2}+\left(\frac{x^{1}}{x^{2}+x^{5}}-x^{2}\frac{x^{1}+x^{4}}{(x^{2}+x^{5})^{2}}\right)\bm{e}_{5},\\
            \bm{a}_{2}(x)
            &:=
            \left(-\frac{x^{3}}{x^{2}+x^{5}}-x^{5}\frac{x^{3}+x^{6}}{(x^{2}+x^{5})^{2}}\right)\bm{e}_{2}+\bm{e}_{3}+\left(\frac{x^{3}}{x^{2}+x^{5}}-x^{2}\frac{x^{3}+x^{6}}{(x^{2}+x^{5})^{2}}\right)\bm{e}_{5},\\
            \bm{a}_{3}(x)
            &:=
            \left(\frac{x^{4}}{x^{2}+x^{5}}-x^{5}\frac{x^{1}+x^{4}}{(x^{2}+x^{5})^{2}}\right)\bm{e}_{2}+\bm{e}_{4}+\left(-\frac{x^{4}}{x^{2}+x^{5}}-x^{2}\frac{x^{1}+x^{4}}{(x^{2}+x^{5})^{2}}\right)\bm{e}_{5},\\
            \bm{a}_{4}(x)
            &:=
            \left(\frac{x^{6}}{x^{2}+x^{5}}-x^{5}\frac{x^{3}+x^{6}}{(x^{2}+x^{5})^{2}}\right)\bm{e}_{2}+\left(-\frac{x^{6}}{x^{2}+x^{5}}-x^{2}\frac{x^{3}+x^{6}}{(x^{2}+x^{5})^{2}}\right)\bm{e}_{5}+\bm{e}_{6}.
        \end{align*}
        Then $\bm{a}_{i}(x)\in T_{x}M\ (i\in \{1, 2, 3, 4\})$ holds.
        The shape operator $A_{x}$ at $x\in M$ satisfies
        \begin{align*}
            \mathcal{B}_{x}
            &:=
            \begin{bmatrix}
                -\bm{a}_{1}(x)+\bm{a}_{3}(x)&\bm{a}_{1}(x)&-\bm{a}_{2}(x)+\bm{a}_{4}(x)&\bm{a}_{2}(x)
            \end{bmatrix},
            &
            A_{x}
            \mathcal{B}_{x}
            &=
            \mathcal{B}_{x}
            \left(J_{2}(-1)\oplus J_{2}(-1)\right)
            .
        \end{align*}
        Moreover, the shape operator $A_{\bm{e}_{5}}$ satisfies
        \begin{align*}
            \mathfrak{B}_{\bm{e}_{5}}
            &:=
            \begin{bmatrix}
                -\bm{a}_{1}(\bm{e}_{5})+\bm{a}_{3}(\bm{e}_{5})&{\displaystyle \frac{\bm{a}_{1}(\bm{e}_{5})+\bm{a}_{3}(\bm{e}_{5})}{2}}&-\bm{a}_{2}(\bm{e}_{5})+\bm{a}_{4}(\bm{e}_{5})&{\displaystyle \frac{\bm{a}_{2}(\bm{e}_{5})+\bm{a}_{4}(\bm{e}_{5})}{2}}
            \end{bmatrix}\\
            &=:
            \begin{bmatrix}
                \bm{b}_{1} & \bm{b}_{2} & \bm{b}_{3} & \bm{b}_{4}
            \end{bmatrix}
            ,\\
            A_{\bm{e}_{5}}
            \mathfrak{B}_{\bm{e}_{5}}
            &=
            \mathfrak{B}_{\bm{e}_{5}}
            \left(J_{2}(-1)\oplus J_{2}(-1)\right)
            ,\quad
            (\langle \bm{b}_{i}, \bm{b}_{j}\rangle_{3})
            =
            E^{\mathfrak{a}}_{2}\oplus E^{\mathfrak{a}}_{2}
            .
        \end{align*}
        \item \label{exam:j}
        Let $n=5$, $s=3$, $f(\bm{e}_{4})=c=1$, and
        \begin{equation*}
            P
            =
            \begin{bmatrix}
                E_{3}&-E^{\mathfrak{a}}_{3}\\
                E^{\mathfrak{a}}_{3}&E_{3}
            \end{bmatrix}
            .
        \end{equation*}
        Then $\mu_{P}(t)=t^2-2t+2$.
        The shape operator $A$ is diagonalizable.
        The principal curvatures of $M$ are $-\sqrt{-1}$, $-\sqrt{-1}$, $\sqrt{-1}$, and $\sqrt{-1}$.
    \end{enumerate}
\end{example}

\subsection{Other Examples}\label{sec:4-2}
We construct the other examples \ref{exam:k}, \ref{exam:l}, and \ref{exam:m} thanks to \cite[3. Examples]{MR783023}.
Although the examples in \cite[3. Examples]{MR783023} are ones of Lorentzian isoparametric hypersurfaces of $\mathbb{R}^{n+1}_{1}$, we can modify them to construct isoparametric hypersurfaces of index 2.

\begin{example}
    \hspace{2em}
    \begin{enumerate}[label=(\alph*)]
        \setcounter{enumi}{10}
        \item \label{exam:k}
        Let
        \begin{align*}
            &X(s)
            :=
            \begin{pmatrix}
                {\displaystyle \frac{\sqrt{2}(s^2+6)}{8}+\frac{1}{2}}\vspace{1mm}\\
                {\displaystyle \frac{\sqrt{2}}{2}}\vspace{1mm}\\
                {\displaystyle \frac{\sqrt{2}(s^2+6)}{8}-\sqrt{2}+\frac{1}{2}}\vspace{1mm}\\
                {\displaystyle -\frac{\sqrt{2}}{2}s}\vspace{1mm}\\
                {\displaystyle 1+\frac{\sqrt{2}}{2}}
            \end{pmatrix}
            ,
            \quad
            Y(s)
            :=
            \begin{pmatrix}
                {\displaystyle \frac{\sqrt{2}(s^2+6)}{8}-\frac{1}{2}}\vspace{1mm}\\
                {\displaystyle \frac{\sqrt{2}}{2}}\vspace{1mm}\\
                {\displaystyle \frac{\sqrt{2}(s^2+6)}{8}-\sqrt{2}-\frac{1}{2}}\vspace{1mm}\\
                {\displaystyle -\frac{\sqrt{2}}{2}s}\vspace{1mm}\\
                {\displaystyle 1-\frac{\sqrt{2}}{2}}
            \end{pmatrix}
            ,\\
            \quad
            &Z(s)
            :=
            \begin{pmatrix}
                {\displaystyle \frac{s}{2}}\vspace{1mm}\\
                0\vspace{1mm}\\
                {\displaystyle \frac{s}{2}}\vspace{1mm}\\
                -1\vspace{1mm}\\
                0
            \end{pmatrix}
            ,
            \quad
            V
            :=
            \begin{pmatrix}
                {\displaystyle \frac{1}{2}}\vspace{1mm}\\
                -1\vspace{1mm}\\
                {\displaystyle \frac{1}{2}}\vspace{1mm}\\
                0\vspace{1mm}\\
                0
            \end{pmatrix}
            ,
            \quad
            C(s)
            :=
            \begin{pmatrix}
                {\displaystyle \frac{s^{2}}{4}+1}\vspace{1mm}\\
                1\vspace{1mm}\\
                {\displaystyle \frac{s^{2}}{4}-1}\vspace{1mm}\\
                -s\vspace{1mm}\\
                \sqrt{2}
            \end{pmatrix}
            ,
            \quad
            x(s):=\int_{0}^{s}X(w)dw
            .
        \end{align*}
        We take an arbitrary $a>0$.
        Let $f\colon \mathbb{R}^{4}\to \mathbb{R}^{5}_{2}$ be the mapping
        \begin{equation*}
            f(\bm{p})
            :=
            x(s)+uY(s)+zZ(s)+vV+\frac{1}{a}\left(1-\sqrt{1+a^{2}v^{2}}\right)C(s)
            ,\quad \bm{p}=(s, u, z, v)\in \mathbb{R}^{4}.
        \end{equation*}
        Let
        \begin{equation*}
            D
            :=
            \left\{
                \begin{pmatrix}
                    s\\
                    u\\
                    z\\
                    v
                \end{pmatrix}
                \in \mathbb{R}^{4}
                \mid
                z+\sqrt{2}\neq 0
            \right\}
            .
        \end{equation*}
        Then the connected components $M$ of $f(D)$ are isoparametric.
        For $\bm{p}\in D$, we have
        \begin{equation*}
            \xi_{f(\bm{p})}
            =
            -\frac{\sqrt{2}z}{z+\sqrt{2}}\sqrt{1+a^{2}v^{2}}Y(s)-avV+\sqrt{1+a^{2}v^{2}}C(s)
            .
        \end{equation*}
        Thus, $\nu=1$.
        Let
        \begin{align*}
            \bm{b}_{1}(\bm{p})
            &:=
            {\displaystyle \frac{\partial f}{\partial u}(\bm{p})}
            ,&
            \bm{b}_{2}(\bm{p})
            &:=
            {\displaystyle \frac{(z+\sqrt{2})^{2}}{2\sqrt{1+a^{2}v^{2}}}\frac{\partial f}{\partial z}(\bm{p})}
            ,\\
            \bm{b}_{3}(\bm{p})
            &:=
            {\displaystyle -\frac{(z+\sqrt{2})^{3}}{2\sqrt{2}(1+a^{2}v^{2})}\frac{\partial f}{\partial s}(\bm{p})}
            ,&
            \bm{b}_{4}(\bm{p})
            &:=
            {\displaystyle \frac{\sqrt{2}azv}{(z+\sqrt{2})\sqrt{1+a^{2}v^{2}}}\frac{\partial f}{\partial u}(\bm{p})+\frac{\partial f}{\partial v}(\bm{p})}
            .
        \end{align*}
        The shape operator $A_{f(\bm{p})}$  at $f(\bm{p})\in M$ satisfies
        \begin{align*}
            \mathcal{B}_{\bm{p}}
            &:=
            \begin{bmatrix}
                \bm{b}_{1}(\bm{p})&\bm{b}_{2}(\bm{p})&\bm{b}_{3}(\bm{p})&\bm{b}_{4}(\bm{p})
            \end{bmatrix}
            ,&
            A_{f(\bm{p})}\mathcal{B}_{\bm{p}}
            &=
            \mathcal{B}_{\bm{p}}
            \left(J_{3}(0)\oplus J_{1}(a)\right)
            .
        \end{align*}
        Then we have
        \begin{equation*}
            \langle \bm{b}_{1}(\bm{p}), \bm{b}_{4}(\bm{p})\rangle_{2}
            =
            \langle \bm{b}_{2}(\bm{p}), \bm{b}_{4}(\bm{p})\rangle_{2}
            =
            \langle \bm{b}_{3}(\bm{p}), \bm{b}_{4}(\bm{p})\rangle_{2}
            =
            0
            ,\quad
            \langle \bm{b}_{4}(\bm{p}), \bm{b}_{4}(\bm{p})\rangle_{2}
            =
            -\frac{1}{1+a^{2}v^{2}}
            <
            0
            .
        \end{equation*}
        \item \label{exam:l}
        Let
        \begin{align*}
            &X(s)
            :=
            \begin{pmatrix}
                {\displaystyle \frac{\sqrt{2}(s^2+2)}{8}-\sqrt{2}}\vspace{1mm}\\
                {\displaystyle \frac{\sqrt{2}}{2}s}\vspace{1mm}\\
                {\displaystyle \frac{\sqrt{2}(s^2+2)}{8}}\vspace{1mm}\\
                {\displaystyle \frac{\sqrt{2}}{2}}\vspace{1mm}\\
                {\displaystyle \frac{\sqrt{2}}{2}}
            \end{pmatrix}
            ,
            \quad
            Y(s)
            :=
            \begin{pmatrix}
                {\displaystyle -\frac{\sqrt{2}(s^2+2)}{8}+\sqrt{2}}\vspace{1mm}\\
                {\displaystyle -\frac{\sqrt{2}}{2}s}\vspace{1mm}\\
                {\displaystyle -\frac{\sqrt{2}(s^2+2)}{8}}\vspace{1mm}\\
                {\displaystyle -\frac{\sqrt{2}}{2}}\vspace{1mm}\\
                {\displaystyle \frac{\sqrt{2}}{2}}
            \end{pmatrix}
            ,\\
            \quad
            &Z(s)
            :=
            \begin{pmatrix}
                {\displaystyle \frac{s}{2}}\vspace{1mm}\\
                1\vspace{1mm}\\
                {\displaystyle \frac{s}{2}}\vspace{1mm}\\
                0\vspace{1mm}\\
                0
            \end{pmatrix}
            ,
            \quad
            V
            :=
            \begin{pmatrix}
                {\displaystyle \frac{1}{2}}\vspace{1mm}\\
                0\vspace{1mm}\\
                {\displaystyle \frac{1}{2}}\vspace{1mm}\\
                -1\vspace{1mm}\\
                0
            \end{pmatrix}
            ,
            \quad
            C(s)
            :=
            \begin{pmatrix}
                {\displaystyle \frac{s^{2}}{4}-1}\vspace{1mm}\\
                s\vspace{1mm}\\
                {\displaystyle \frac{s^{2}}{4}+1}\vspace{1mm}\\
                1\vspace{1mm}\\
                0
            \end{pmatrix}
            ,
            \quad
            x(s):=\int_{0}^{s}X(w)dw
            .
        \end{align*}
        We take an arbitrary $a>0$.
        Let $f\colon \mathbb{R}^{4}\to \mathbb{R}^{5}_{2}$ be the mapping
        \begin{equation*}
            f(\bm{p})
            :=
            x(s)+uY(s)+zZ(s)+vV+\frac{1}{a}\left(1-\sqrt{1-a^{2}v^{2}}\right)C(s)
            ,\quad \bm{p}=(s, u, z, v)\in \mathbb{R}^{4}.
        \end{equation*}
        Let
        \begin{equation*}
            D
            :=
            \left\{
                \begin{pmatrix}
                    s\\
                    u\\
                    z\\
                    v
                \end{pmatrix}
                \in \mathbb{R}^{4}
                \mid
                z-\sqrt{2}\neq 0
                ,
                v\in \left(-\frac{1}{a}, \frac{1}{a}\right)
            \right\}
            .
        \end{equation*}
        Then the connected components $M$ of $f(D)$ are isoparametric.
        For $\bm{p}\in D$, we have
        \begin{equation*}
            \xi_{f(\bm{p})}
            =
            \frac{\sqrt{2}z}{z-\sqrt{2}}\sqrt{1-a^{2}v^{2}}Y(s)-avV+\sqrt{1-a^{2}v^{2}}C(s)
            .
        \end{equation*}
        Thus, $\nu=1$.
        Let
        \begin{align*}
            \bm{b}_{1}
            &:=
            {\displaystyle \frac{\partial f}{\partial u}(\bm{p})}
            ,&
            \bm{b}_{2}
            &:=
            {\displaystyle \frac{(z-\sqrt{2})^{2}}{2\sqrt{1-a^{2}v^{2}}}\frac{\partial f}{\partial z}(\bm{p})}
            ,\\
            \bm{b}_{3}
            &:=
            {\displaystyle \frac{(z-\sqrt{2})^{3}}{2\sqrt{2}(1-a^{2}v^{2})}\frac{\partial f}{\partial s}(\bm{p})}
            ,&
            \bm{b}_{4}
            &:=
            {\displaystyle \frac{\sqrt{2}azv}{(z-\sqrt{2})\sqrt{1-a^{2}v^{2}}}\frac{\partial f}{\partial u}(\bm{p})+\frac{\partial f}{\partial v}(\bm{p})}
            .
        \end{align*}
        The shape operator $A_{f(\bm{p})}$ at $\bm{p}\in D$ satisfies
        \begin{align*}
            \mathcal{B}_{\bm{p}}
            &:=
            \begin{bmatrix}
                \bm{b}_{1}(\bm{p})&\bm{b}_{2}(\bm{p})&\bm{b}_{3}(\bm{p})&\bm{b}_{4}(\bm{p})
            \end{bmatrix}
            ,&
            A_{f(\bm{p})}\mathcal{B}_{\bm{p}}
            =
            \mathcal{B}_{\bm{p}}
            \left(J_{3}(0)\oplus J_{1}(a)\right)
            .
        \end{align*}
        Then we have
        \begin{equation*}
            \langle \bm{b}_{1}(\bm{p}), \bm{b}_{4}(\bm{p})\rangle_{2}
            =
            \langle \bm{b}_{2}(\bm{p}), \bm{b}_{4}(\bm{p})\rangle_{2}
            =
            \langle \bm{b}_{3}(\bm{p}), \bm{b}_{4}(\bm{p})\rangle_{2}
            =
            0
            ,\quad
            \langle \bm{b}_{4}(\bm{p}), \bm{b}_{4}(\bm{p})\rangle_{2}
            =
            \frac{1}{1-a^{2}v^{2}}
            >
            0
            .
        \end{equation*}
    \end{enumerate}
\end{example}
Example \ref{exam:101} is based on an original idea although we refer to \cite[3. Examples]{MR783023}.
\begin{example}\label{exam:101}
    \hspace{2em}
    \begin{enumerate}[label=(\alph*)]
        \setcounter{enumi}{12}
        \item \label{exam:m}
        Let
        \begin{align*}
            X
            &:=
            \begin{pmatrix}
                0\\
                -1\\
                0\\
                0\\
                1
            \end{pmatrix}
            ,\quad
            Y(u)
            :=
            \begin{pmatrix}
                u\\
                -1\\
                u\\
                1\\
                0
            \end{pmatrix}
            ,\quad
            x(s)=\int_{0}^{s}Xdv
            ,\quad
            y(u)=\int_{0}^{u}Y(v)dv
            ,\\
            Z
            &:=
            \begin{pmatrix}
                1\\
                0\\
                1\\
                0\\
                0
            \end{pmatrix}
            ,\quad
            W(u)
            :=
            \frac{1}{2}
            \begin{pmatrix}
                u^{2}+1\\
                0\\
                u^{2}-1\\
                2u\\
                0
            \end{pmatrix}
            ,\quad
            C(u)
            :=
            \begin{pmatrix}
                -u\\
                1\\
                -u\\
                -1\\
                -1
            \end{pmatrix}
            .
        \end{align*}
        Let $f\colon\mathbb{R}^{4}\to \mathbb{R}^{5}_{2}$ be the mapping
        \begin{equation*}
            f(\bm{p})
            :=
            x(s)+wW(u)+zZ-\frac{z^{2}}{2}C(u)+y(u)
            ,\quad
            \bm{p}=(s, w, z, u)\in \mathbb{R}^{4}
            .
        \end{equation*}
        Then $M:=f(\mathbb{R}^{4})$ is isoparametric.
        For $\bm{p}\in \mathbb{R}^{4}$, we have
        \begin{equation*}
            \xi_{f(\bm{p})}
            =
            \left(-w+\frac{z^{3}}{2}\right)X-zW(u)+C(u)
            .
        \end{equation*}
        Thus, $\nu=1$.
        Let
        \begin{align*}
            \bm{b}_{1}(\bm{p})
            &:=
            {\displaystyle \frac{\partial f}{\partial s}(\bm{p})}
            ,&
            \bm{b}_{2}(\bm{p})
            &:=
            {\displaystyle \frac{\partial f}{\partial w}(\bm{p})}
            ,\\
            \bm{b}_{3}(\bm{p})
            &:=
            {\displaystyle \frac{3z^{2}}{2}\frac{\partial f}{\partial w}(\bm{p})+\frac{\partial f}{\partial z}(\bm{p})}
            ,&
            \bm{b}_{4}(\bm{p})
            &:=
            {\displaystyle \left(\frac{9z^{4}}{4}+z\right)\frac{\partial f}{\partial w}(\bm{p})+\frac{3z^{2}}{2}\frac{\partial f}{\partial z}(\bm{p})+\frac{\partial f}{\partial u}(\bm{p})}
            .
        \end{align*}
        The shape operator $A_{f(\bm{p})}$ at $\bm{p}\in \mathbb{R}^{4}$ satisfies
        \begin{align*}
            \mathcal{B}_{\bm{p}}
            &:=
            \begin{bmatrix}
                \bm{b}_{1}(\bm{p})&\bm{b}_{2}(\bm{p})&\bm{b}_{3}(\bm{p})&\bm{b}_{4}(\bm{p})
            \end{bmatrix}
            ,&
            A_{f(\bm{p})}\mathcal{B}_{\bm{p}}
            =
            \mathcal{B}_{\bm{p}}
            J_{4}(0)
            .
        \end{align*}
    \end{enumerate}
\end{example}

\subsection{Non-existence}\label{sec:4-3}
We show the non-existence of an isoparametric hypersurface of index 2 whose Petrov type is type I, II, III, or IV of index 2 in a certain case.
We refer to \cite[Theorem 4.10]{MR783023} to prove the following theorems by contradiction.
\subsubsection{Type I}\label{sec:4-3-1}
\begin{theorem}
    \label{theo:ne-1}
    There exist no isoparametric hypersurfaces whose Petrov type is type I of index 2 in $\mathbb{R}^{5}_{2}$ or $S^{5}_{2}$.
\end{theorem}
Suppose that there exists an isoparametric hypersurface $M$ whose Petrov type is type I of index 2 in $\mathbb{R}^{5}_{2}$, $S^{5}_{2}$, or $S^{5}_{3}$.
By Proposition \ref{prop:1}, there exists a local frame field $C_{1}, C_{2}, C_{3}, C_{4}$ such that
\begin{equation*}
    AC_{1}=\alpha C_{1}+\beta C_{2}, \quad
    AC_{2}=-\beta C_{1}+\alpha C_{2},\quad
    AC_{3}=C_{1}+\alpha C_{3}+\beta C_{4},\quad
    AC_{4}=C_{2}-\beta C_{3}+\alpha C_{4},
\end{equation*}
\begin{equation*}
    \langle C_{i}, C_{j}\rangle
    =
    \begin{cases}
        (-1)^{i}, & (\{i, j\}=\{1, 3\}\text{ or }\{i, j\}=\{2, 4\}),\\
        0, & (\text{otherwise}).
    \end{cases}
\end{equation*}
We now investigate the properties of the local frame field.
For $i, j, k\in \{1, 2, 3, 4\}$, we have $\langle \nabla_{C_{i}}C_{j}, C_{k}\rangle=-\langle \nabla_{C_{i}}C_{k}, C_{j}\rangle$ since $C_{i}\langle C_{j}, C_{k}\rangle=0$, that is, $\nabla$ preserves $\langle\cdot, \cdot\rangle$.
In particular, when $j=k$, we have $\langle \nabla_{C_{i}}C_{j}, C_{j}\rangle=0$.
\begin{lemma}
    \label{lemm:2}
    We have $\langle \nabla_{C_{1}}C_{1}, C_{2}\rangle=\langle \nabla_{C_{2}}C_{1}, C_{2}\rangle=0$.
\end{lemma}
\begin{proof}
    Lemma \ref{lemm:2} follows from $\{C_{1}, C_{2}\}C_{2}$ and $\{C_{1}, C_{2}\}C_{1}$.
\end{proof}
\begin{remark}\label{rem:1}
    Lemma \ref{lemm:2} is a claim in Section \ref{sec:4-3-1}.
    However, Lemma \ref{lemm:2} also holds in Sections \ref{sec:4-3-2}, \ref{sec:4-3-3}, or \ref{sec:4-3-4}.
\end{remark}
\begin{lemma}
    \label{lemm:4}
    We have $\langle \nabla_{C_{1}}C_{3}, C_{4}\rangle=\langle \nabla_{C_{2}}C_{3}, C_{4}\rangle=0$.
\end{lemma}
\begin{proof}
    We have
    \begin{align}
        &\{C_{1}, C_{3}\}C_{3};
        &
        2\langle \nabla_{C_{1}}C_{1}, C_{3}\rangle
        -
        2\beta\langle \nabla_{C_{1}}C_{3}, C_{4}\rangle
        &=
        \beta\langle\nabla_{C_{3}}C_{2}, C_{3}\rangle
        -
        \beta\langle\nabla_{C_{3}}C_{1}, C_{4}\rangle
        ,\label{eq:9}\\
        &\{C_{1}, C_{4}\}C_{4};
        &
        2\langle \nabla_{C_{1}}C_{2}, C_{4}\rangle
        -
        2\beta\langle \nabla_{C_{1}}C_{3}, C_{4}\rangle
        &=
        \beta\langle\nabla_{C_{4}}C_{2}, C_{4}\rangle\notag\\
        &&&\quad-
        \langle \nabla_{C_{4}}C_{1}, C_{2}\rangle
        +
        \beta\langle\nabla_{C_{4}}C_{1}, C_{3}\rangle
        ,\label{eq:10}\\
        &\{C_{2}, C_{3}\}C_{4};
        &
        \langle \nabla_{C_{2}}C_{1}, C_{4}\rangle
        +
        \langle \nabla_{C_{2}}C_{2}, C_{3}\rangle
        &=
        -
        \beta\langle\nabla_{C_{3}}C_{1}, C_{4}\rangle
        +
        \beta\langle \nabla_{C_{3}}C_{2}, C_{3}\rangle
        ,\label{eq:11}\\
        &\{C_{2}, C_{4}\}C_{3};
        &
        \langle \nabla_{C_{2}}C_{2}, C_{3}\rangle
        +
        \langle \nabla_{C_{2}}C_{1}, C_{4}\rangle
        &=
        -
        \beta\langle \nabla_{C_{4}}C_{1}, C_{3}\rangle\notag\\
        &&&\quad
        +
        \langle\nabla_{C_{4}}C_{1}, C_{2}\rangle
        -
        \beta\langle \nabla_{C_{4}}C_{2}, C_{4}\rangle
        ,\label{eq:12}\\
        &\{C_{1}, C_{3}\}C_{2};
        &
        \langle \nabla_{C_{1}}C_{1}, C_{2}\rangle
        -
        \beta\langle \nabla_{C_{1}}C_{2}, C_{4}\rangle
        &-
        \beta\langle \nabla_{C_{1}}C_{1}, C_{3}\rangle
        =
        0
        .
        \label{eq:13}
    \end{align}
    We note that $\beta\neq 0$ and Lemma \ref{lemm:2}.
    We have $\langle \nabla_{C_{1}}C_{3}, C_{4}\rangle=0$ by solving the system of linear equations \eqref{eq:9}, \eqref{eq:10}, \eqref{eq:11}, \eqref{eq:12}, and \eqref{eq:13}.
    Similarly, we have $\langle \nabla_{C_{2}}C_{3}, C_{4}\rangle=0$ by solving the system of linear equations $\{C_{2}, C_{3}\}C_{3}$, $\{C_{2}, C_{4}\}C_{4}$, $\{C_{1}, C_{3}\}C_{4}$, $\{C_{1}, C_{4}\}C_{3}$, and $\{C_{2}, C_{4}\}C_{1}$.
\end{proof}
\begin{proposition}
    \label{prop:3}
    We have $\kappa+\nu(\alpha^{2}+\beta^{2})=0$.
\end{proposition}
\begin{proof}
    Using the Gauss equation, we have $\langle R(C_{1}, C_{2})C_{3}, C_{4}\rangle=\kappa+\nu(\alpha^{2}+\beta^{2})$.
    By Lemmas \ref{lemm:2} and \ref{lemm:4}, and \eqref{eq:R}, we have
    \begin{align*}
        \nabla_{C_{2}}C_{3}&=-\langle\nabla_{C_{2}}C_{3}, C_{1}\rangle C_{3}+\langle\nabla_{C_{2}}C_{3}, C_{2}\rangle C_{4},&\langle \nabla_{C_{1}}(\nabla_{C_{2}}C_{3}), C_{4}\rangle&=0,\\
        \nabla_{C_{1}}C_{3}&=-\langle\nabla_{C_{1}}C_{3}, C_{1}\rangle C_{3}+\langle\nabla_{C_{1}}C_{3}, C_{2}\rangle C_{4},&-\langle \nabla_{C_{2}}(\nabla_{C_{1}}C_{3}), C_{4}\rangle&=0,\\
        \nabla_{C_{1}}C_{2}&=-\langle\nabla_{C_{1}}C_{2}, C_{3}\rangle C_{1}+\langle\nabla_{C_{1}}C_{2}, C_{4}\rangle C_{2},&-\langle \nabla_{\nabla_{C_{1}}C_{2}}C_{3}, C_{4}\rangle&=0,\\
        \nabla_{C_{2}}C_{1}&=-\langle\nabla_{C_{2}}C_{1}, C_{3}\rangle C_{1}+\langle\nabla_{C_{2}}C_{1}, C_{4}\rangle C_{2},&\langle \nabla_{\nabla_{C_{2}}C_{1}}C_{3}, C_{4}\rangle&=0.
    \end{align*}
\end{proof}
We can prove Theorem \ref{theo:ne-1} by Proposition \ref{prop:3} in $\delta=0$ or $\delta=1$.
\begin{corollary}
    If $\lambda$ is a complex principal curvature of an isoparametric hypersurface whose Petrov type is type I of index 2 in $S^{5}_{3}$, then $|\lambda|=1$ holds.
\end{corollary}

\subsubsection{Type II}\label{sec:4-3-2}
\begin{theorem}
    \label{theo:ne-2}
    There exist no isoparametric hypersurfaces whose Petrov type is type II of index 2 in $\mathbb{R}^{5}_{2}$ or $S^{5}_{2}$ in the case where the multiplicities of the complex principal curvatures are $2$.
\end{theorem}
Suppose that there exists an isoparametric hypersurface $M$ whose Petrov type is type II of index 2 in $\mathbb{R}^{5}_{2}$, $S^{5}_{2}$ or $S^{5}_{3}$.
By Proposition \ref{prop:1}, there exists a local frame field $C_{1}, C_{2}, D_{1}, D_{2}$ such that
\begin{equation*}
    AC_{1}=\alpha_{1} C_{1}+\beta_{1} C_{2},\quad
    AC_{2}=-\beta_{1} C_{1}+\alpha_{1} C_{2},\quad
    AD_{1}=\alpha_{2} C_{3}+\beta_{2} C_{4},\quad
    AD_{2}=-\beta_{2} C_{3}+\alpha_{2} C_{4},\\
\end{equation*}
\begin{equation*}
    \langle C_{i}, C_{j}\rangle
    =
    \langle D_{i}, D_{j}\rangle
    =
    \begin{cases}
        (-1)^{i}, & (i=j),\\
        0, & (\text{otherwise}).
    \end{cases}
\end{equation*}
Let us introduce the following notation:
\begin{align*}
    a_{1}^{i}
    &:=
    \langle \nabla_{C_{1}}C_{1}, D_{i}\rangle,
    &
    a_{2}^{i}
    &:=
    \langle \nabla_{C_{1}}C_{2}, D_{i}\rangle,
    &
    a_{3}^{i}
    &:=
    \langle \nabla_{C_{2}}C_{1}, D_{i}\rangle,
    &
    a_{4}^{i}
    &:=
    \langle \nabla_{C_{2}}C_{2}, D_{i}\rangle,
    &
    i\in \{1, 2\}.
\end{align*}
\begin{lemma}
    \label{lemm:11}
    If $\alpha_{2}=\alpha_{1}$ and $\beta_{2}=\beta_{1}$, we have $\langle \nabla_{D_{1}}C_{1}, C_{2}\rangle=\langle \nabla_{D_{2}}C_{1}, C_{2}\rangle=0$ and $a_{3}^{1}a_{2}^{1}-a_{3}^{2}a_{2}^{2}-a_{1}^{1}a_{4}^{1}+a_{1}^{2}a_{4}^{2}=0$.
\end{lemma}
\begin{proof}
    We have
    \begin{align}
        &\{D_{1}, C_{1}\}C_{1};
        &
        -2\beta_{1}\langle\nabla_{D_{1}}C_{1}, C_{2}\rangle
        &=
        (\alpha_{1}-\alpha_{2})a_{1}^{1}-\beta_{2} a_{1}^{2}+\beta_{1} a_{2}^{1},\label{eq:50}\\
        &\{D_{1}, C_{1}\}C_{2};
        &
        0
        &=
        (\alpha_{1}-\alpha_{2})a_{2}^{1}-\beta_{2} a_{2}^{2}-\beta_{1} a_{1}^{1},\label{eq:51}\\
        &\{D_{1}, C_{2}\}C_{1};
        &
        0
        &=
        (\alpha_{1}-\alpha_{2})a_{3}^{1}-\beta_{2} a_{3}^{2}+\beta_{1} a_{4}^{1},\label{eq:52}\\
        &\{D_{2}, C_{1}\}C_{2};
        &
        0
        &=
        (\alpha_{1}-\alpha_{2})a_{2}^{2}+\beta_{2} a_{2}^{1}-\beta_{1} a_{1}^{2},\label{eq:55}\\
        &\{D_{2}, C_{2}\}C_{1};
        &
        0
        &=
        (\alpha_{1}-\alpha_{2})a_{3}^{2}+\beta_{2} a_{3}^{1}+\beta_{1} a_{4}^{2},\label{eq:56}\\
        &\{D_{2}, C_{2}\}C_{2};
        &
        -2\beta_{1}\langle\nabla_{D_{2}}C_{1}, C_{2}\rangle
        &=
        (\alpha_{1}-\alpha_{2})a_{4}^{2}+\beta_{2} a_{4}^{1}-\beta_{1} a_{3}^{2}.\label{eq:57}
    \end{align}
    We note the assumption.
    We should solve the system of linear equations \eqref{eq:50}, \eqref{eq:51}, \eqref{eq:52}, \eqref{eq:55}, \eqref{eq:56}, and \eqref{eq:57}.
\end{proof}
\begin{proposition}
    \label{prop:7}
    If $\alpha_{2}=\alpha_{1}$ and $\beta_{2}=\beta_{1}$, we have $\kappa+\nu((\alpha_{1})^{2}+(\beta_{1})^{2})=0$.
\end{proposition}
\begin{proof}
    Using the Gauss equation, we have $\langle R(C_{1}, C_{2})C_{1}, C_{2}\rangle=\kappa+\nu((\alpha_{1})^{2}+(\beta_{1})^{2})$.
    We note Remark \ref{rem:1}.
    By Lemmas \ref{lemm:2} and \ref{lemm:11}, and \eqref{eq:R}, we have $\langle R(C_{1}, C_{2})C_{1}, C_{2}\rangle=0$.
\end{proof}
We can prove Theorem \ref{theo:ne-2} by Proposition \ref{prop:7}.
\begin{corollary}\label{coro:2}
    If $\lambda$ is a complex principal curvature of an isoparametric hypersurface whose Petrov type is type II of index 2 in $S^{5}_{3}$ in the case where the multiplicities of the complex principal curvatures are $2$, then $|\lambda|=1$ holds.
\end{corollary}
One of the examples of Corollary \ref{coro:2} is \ref{exam:j}.

\subsubsection{Type III}\label{sec:4-3-3}
\begin{theorem}\label{theo:ne-3}
    There exist no isoparametric hypersurfaces whose Petrov type is type III of index 2 in $\mathbb{R}^{5}_{2}$.
\end{theorem}
Suppose that there exists an isoparametric hypersurface $M$ whose Petrov type is type III of index 2 in $\mathbb{R}^{5}_{2}$.
By Proposition \ref{prop:1}, there exists a local frame field $C_{1}, C_{2}, X, Y$ such that
\begin{equation*}
    AC_{1}=\alpha_{1} C_{1}+\beta_{1} C_{2},\quad
    AC_{2}=-\beta_{1} C_{1}+\alpha_{1} C_{2},\quad
    AX=a_{0}X,\quad
    AY=a_{1}Y,\\
\end{equation*}
\begin{equation*}
    \langle C_{i}, C_{j}\rangle
    =
    \begin{cases}
        (-1)^{i}, & (i=j),\\
        0, & (\text{otherwise}),
    \end{cases}
    \quad
    \langle X, X\rangle=-1,\quad
    \langle Y, Y\rangle=1,\quad
    \langle C_{i}, X\rangle
    =
    \langle C_{i}, Y\rangle
    =
    \langle X, Y\rangle
    =
    0.
\end{equation*}
Since the ambient space of $M$ is $\mathbb{R}^{5}_{2}$, $\kappa=0$ and $\nu=1$.
\begin{proposition}
    \label{prop:4}
    $M$ has at most one non-zero real principal curvature, that is, one of the following cases holds:
    \begin{enumerate}[label=(\arabic*)]
        \item $a_{0}=a_{1}=0.$\label{prop:4-1}
        \item There exists $r\neq 0$ such that $a_{0}=a_{1}=r$.\label{prop:4-2}
        \item $a_{0}=0$ and there exists $r\neq 0$ such that $a_{1}=r$.\label{prop:4-3}
        \item $a_{1}=0$ and there exists $r\neq 0$ such that $a_{0}=r$.\label{prop:4-4}
    \end{enumerate}
\end{proposition}
\begin{proof}
    Similar to \cite[Corollary 2.9]{MR783023}.
\end{proof}
\begin{lemma}
    \label{lemm:6}
    In \ref{prop:4-2}, \ref{prop:4-3}, or \ref{prop:4-4} of Proposition \ref{prop:4}, we have $\alpha\neq 0$ and
    \begin{equation}
        r=\frac{\alpha^{2}+\beta^{2}}{\alpha}.
        \label{eq:18}
    \end{equation}
\end{lemma}
\begin{proof}
    Using Cartan's formula \eqref{eq:Cartan} for $r$, by Proposition \ref{prop:4}, we have
    \begin{equation*}
        2r
        \left(
            \frac{(\alpha-r)\alpha+\beta^{2}}{(\alpha-r)^{2}+\beta^{2}}
        \right)
        =
        0.
    \end{equation*}
    Hence, we have $(r-\alpha)\alpha-\beta^{2}=0$.
    If $\alpha=0$, then $\beta=0$, which is a contradiction.
    Thus, we have $\alpha\neq 0$ and \eqref{eq:18}.
\end{proof}
\begin{lemma}\label{lemm:101}
    For $Z=X, Y$, we have $\langle \nabla_{Z}Z, C_{1}\rangle=\langle \nabla_{Z}Z, C_{2}\rangle=0$.
\end{lemma}
\begin{proof}
    We have
    \begin{align}
        &\{C_{1}, Z\}Z;
        &
        0&=-(\alpha-a_{Z})\langle \nabla_{Z}Z, C_{1}\rangle-\beta\langle \nabla_{Z}Z, C_{2}\rangle,\label{eq:103}\\
        &\{C_{2}, Z\}Z;
        &
        0&=-(\alpha-a_{Z})\langle \nabla_{Z}Z, C_{2}\rangle+\beta\langle \nabla_{Z}Z, C_{1}\rangle,\label{eq:104}
    \end{align}
    where $a_{Z}$ is $a_{0}$ or $a_{1}$ if $Z=X$ or $Z=Y$, respectively.
    We should solve the system of linear equations \eqref{eq:103} and \eqref{eq:104}.
\end{proof}
\begin{lemma}
    \label{lemm:3}
    For $Z=X, Y$, we have
    \begin{align*}
        \langle \nabla_{C_{1}}C_{1}, Z\rangle
        &=
        \langle \nabla_{C_{2}}C_{2}, Z\rangle,\quad
        \langle \nabla_{C_{1}}C_{2}, Z\rangle
        =
        -
        \langle \nabla_{C_{2}}C_{1}, Z\rangle
        =
        \frac{-2\beta^{2}}{(\alpha-a_{Z})^{2}+\beta^{2}}\langle \nabla_{Z}C_{1}, C_{2}\rangle,
    \end{align*}
    where $a_{Z}$ is $a_{0}$ or $a_{1}$ if $Z=X$ or $Z=Y$, respectively.
\end{lemma}
\begin{proof}
    We have
    \begin{align}
        &\{C_{1}, Z\}C_{1};
        &
        (\alpha-a_{Z})\langle \nabla_{C_{1}}C_{1}, Z\rangle+\beta\langle \nabla_{C_{1}}C_{2}, Z\rangle &=-2\beta\langle \nabla_{Z}C_{1}, C_{2}\rangle,\label{eq:3}\\
        &\{C_{1}, Z\}C_{2};
        &
        (\alpha-a_{Z})\langle \nabla_{C_{1}}C_{2}, Z\rangle-\beta\langle \nabla_{C_{1}}C_{1}, Z\rangle &=0,\label{eq:4}\\
        &\{C_{2}, Z\}C_{1};
        &
        (\alpha-a_{Z})\langle \nabla_{C_{2}}C_{1}, Z\rangle+\beta\langle \nabla_{C_{2}}C_{2}, Z\rangle &=0,\label{eq:5}\\
        &\{C_{2}, Z\}C_{2};
        &
        (\alpha-a_{Z})\langle \nabla_{C_{2}}C_{2}, Z\rangle-\beta\langle \nabla_{C_{2}}C_{1}, Z\rangle &=-2\beta\langle \nabla_{Z}C_{1}, C_{2}\rangle.\label{eq:6}
    \end{align}
    We note that $\beta=0$.
    We should solve the system of linear equation \eqref{eq:3}, \eqref{eq:4}, \eqref{eq:5}, and \eqref{eq:6}.
\end{proof}
Let us introduce the following notation:
\begin{equation*}
    P_{i}:=\langle\nabla_{X}Y, C_{i}\rangle
    ,\quad
    Q_{i}:=\langle\nabla_{Y}X, C_{i}\rangle
    ,\quad
    R_{i}:=\langle\nabla_{C_{i}}X, Y\rangle
    ,\quad
    i\in \{1, 2\}.
\end{equation*}
\begin{proposition}
    \label{prop:103}
    We have
    \begin{equation}
        \begin{split}
            2\alpha a_{1}
            =
            2(\kappa+\nu\alpha a_{1})
            &=
            -2\langle\nabla_{C_{1}}C_{1}, Y\rangle^{2}-2\langle\nabla_{C_{1}}C_{2}, Y\rangle^{2}\\
            &
            -P_{1}Q_{1}
            +Q_{1}R_{1}
            +R_{1}P_{1}
            +P_{2}Q_{2}
            -Q_{2}R_{2}
            -R_{2} P_{2}
            .
        \end{split}
        \label{eq:120}
    \end{equation}
\end{proposition}
\begin{proof}
    We note Remark \ref{rem:1}.
    By Lemmas \ref{lemm:2}, \ref{lemm:101}, and \ref{lemm:3}, and \eqref{eq:R}, we have
    \begin{equation}
        \begin{split}
            \langle R(C_{1}, Y)C_{1}, Y\rangle
            &=
            -\langle\nabla_{C_{1}}C_{1}, Y\rangle^{2}-\langle\nabla_{C_{1}}C_{2}, Y\rangle^{2}
            -
            \langle \nabla_{Y}\nabla_{C_{1}}C_{1}, Y\rangle\\
            &+
            Q_{1}R_{1}
            +
            R_{1}P_{1}
            -
            Q_{1}P_{1}
            .
        \end{split}
        \label{eq:118}
    \end{equation}
    Similarly, we have
    \begin{align}
        \begin{split}
            \langle R(C_{2}, Y)C_{2}, Y\rangle
            &=
            \langle\nabla_{C_{1}}C_{1}, Y\rangle^{2}+\langle\nabla_{C_{1}}C_{2}, Y\rangle^{2}
            -
            \langle \nabla_{Y}\nabla_{C_{1}}C_{1}, Y\rangle\\
            &+
            Q_{2}R_{2}
            +
            R_{2}P_{2}
            -
            Q_{2}P_{2}
            .
        \end{split}
        \label{eq:119}
    \end{align}
    We obtain \eqref{eq:120} by \eqref{eq:118} and \eqref{eq:119}, and using the Gauss equation.
\end{proof}
\begin{lemma}\label{lemm:102}
    If $a_{0}=a_{1}$, then we have $P_{1}=P_{2}=Q_{1}=Q_{2}=0$.
\end{lemma}
\begin{proof}
    By the assumption, we have
    \begin{align}
        &\{C_{1}, X\}Y;
        &
        0&=-(\alpha-a_{1})P_{1}-\beta P_{2},\label{eq:121}\\
        &\{C_{2}, X\}Y;
        &
        0&=-(\alpha-a_{1})P_{2}+\beta P_{1},\label{eq:122}\\
        &\{C_{1}, Y\}X;
        &
        0&=-(\alpha-a_{0})Q_{1}-\beta Q_{2},\label{eq:123}\\
        &\{C_{2}, Y\}X;
        &
        0&=-(\alpha-a_{0})Q_{2}+\beta Q_{1}.\label{eq:124}
    \end{align}
    We should solve the system of linear equations \eqref{eq:121} and \eqref{eq:122} and the system of linear equations \eqref{eq:123} and \eqref{eq:124}.
\end{proof}
\begin{lemma}\label{lemm:201}
    In \ref{prop:4-3} of Proposition \ref{prop:4}, we have
    \begin{align}
        Q_{1}&=\frac{\beta}{\alpha}P_{2}
        ,&
        Q_{2}&=-\frac{\beta}{\alpha}P_{1}
        ,\label{eq:201}\\
        R_{1}&=-\frac{\beta^{2}}{\alpha^{2}+\beta^{2}}P_{1}+\frac{\alpha\beta}{\alpha^{2}+\beta^{2}}P_{2}
        ,&
        R_{2}&=-\frac{\alpha\beta}{\alpha^{2}+\beta^{2}}P_{1}-\frac{\beta^{2}}{\alpha^{2}+\beta^{2}}P_{2}
        .\label{eq:202}
    \end{align}
    In \ref{prop:4-4} of Proposition \ref{prop:4}, we have
    \begin{align}
        Q_{1}&=-\frac{\alpha}{\beta}P_{2}
        ,&
        Q_{2}&=\frac{\alpha}{\beta}P_{1}
        ,\label{eq:203}\\
        R_{1}&=-\frac{\alpha^{2}}{\alpha^{2}+\beta^{2}}P_{1}-\frac{\alpha\beta}{\alpha^{2}+\beta^{2}}P_{2}
        ,&
        R_{2}&=\frac{\alpha\beta}{\alpha^{2}+\beta^{2}}P_{1}-\frac{\alpha^{2}}{\alpha^{2}+\beta^{2}}P_{2}
        .\label{eq:204}
    \end{align}
\end{lemma}
\begin{proof}
    \eqref{eq:201} and \eqref{eq:203} follow from $\{X, Y\}C_{1}$ and $\{X, Y\}C_{2}$.
    Thus, \eqref{eq:202} and \eqref{eq:204} follow from $\{C_{1}, Y\}X$ and $\{C_{2}, Y\}X$.
\end{proof}
\begin{lemma}\label{lemm:204}
    We have $\langle\nabla_{X}X, Y\rangle=0$ and $\langle\nabla_{Y}Y, X\rangle=0$.
\end{lemma}
\begin{proof}
    Lemma \ref{lemm:204} follows from $\{X, Y\}X$ and $\{Y, X\}Y$.
\end{proof}
\begin{lemma}\label{lemm:202}
    If $a_{0}\neq a_{1}$, then we have $P_{1}P_{2}=0$.
\end{lemma}
\begin{proof}
    By Lemmas \ref{lemm:101} and \ref{lemm:204}, and \eqref{eq:R}, we have
    \begin{equation}
        \langle R(Y, X)Y, X\rangle
        =
        P_{1}Q_{1}-Q_{1}R_{1}+R_{1}P_{1}
        -P_{2}Q_{2}+Q_{2}R_{2}-R_{2}P_{2}
        .\label{eq:205}
    \end{equation}
    Using the Gauss equation, we have $\langle R(Y, X)Y, X\rangle=0$.
    In \ref{prop:4-3} of Proposition \ref{prop:4}, the right-hand side of \eqref{eq:205} is equal to $(4\beta/\alpha)P_{1}P_{2}=0$ by Lemma \ref{lemm:201}.
    In \ref{prop:4-4} of Proposition \ref{prop:4}, the right-hand side of \eqref{eq:205} is equal to $(-4\alpha/\beta)P_{1}P_{2}=0$ by Lemma \ref{lemm:201}.
\end{proof}
\begin{lemma}\label{lemm:203}
    We have $\langle\nabla_{C_{1}}C_{2}, Y\rangle\langle\nabla_{Y}C_{1}, C_{2}\rangle\neq 0$.
\end{lemma}
\begin{proof}
    We suppose that $\langle\nabla_{C_{1}}C_{2}, Y\rangle\langle\nabla_{Y}C_{1}, C_{2}\rangle=0$.
    By Lemma \ref{lemm:3} and \eqref{eq:4}, we have
    \begin{equation}
        \langle\nabla_{C_{1}}C_{1}, Y\rangle=\langle\nabla_{C_{1}}C_{2}, Y\rangle=\langle\nabla_{Y}C_{1}, C_{2}\rangle=0
        .\label{eq:211}
    \end{equation}
    Using the Gauss equation, we have $\langle R(C_{1}, C_{2})C_{1}, C_{2}\rangle=\kappa+\nu(\alpha^{2}+\beta^{2})$.
    By Lemmas \ref{lemm:2} and \ref{lemm:3}, \eqref{eq:R}, and \eqref{eq:211}, we have
    \begin{align*}
        \kappa+\nu(\alpha^{2}+\beta^{2})
        &=
        -\langle\nabla_{C_{1}}C_{1}, X\rangle^{2}
        -\langle\nabla_{C_{1}}C_{2}, X\rangle^{2}
        -\frac{4\beta^{2}}{(\alpha-a_{0})^{2}+\beta^{2}}\langle\nabla_{X}C_{1}, C_{2}\rangle^{2}
        ,
    \end{align*}
    which is a contradiction.
\end{proof}
\begin{proof}[(The Proof of Theorem \ref{theo:ne-3})]
    In \ref{prop:4-1} of Proposition \ref{prop:4}, we obtain $\langle \nabla_{C_{1}}C_{1}, Y\rangle=\langle \nabla_{C_{1}}C_{2}, Y\rangle=0$ by Proposition \ref{prop:103} and Lemma \ref{lemm:102}, which is a contradiction by Lemma \ref{lemm:203}.

    In \ref{prop:4-2} of Proposition \ref{prop:4}, we have a contradiction by Proposition \ref{prop:103} and Lemmas \ref{lemm:6} and \ref{lemm:102}.

    In \ref{prop:4-3} of Proposition \ref{prop:4}, by Proposition \ref{prop:103} and Lemmas \ref{lemm:6}, \ref{lemm:201}, and \ref{lemm:203}, we have
    \begin{equation}
        2(\alpha^{2}+\beta^{2})
        =
        -2\langle\nabla_{C_{1}}C_{1}, Y\rangle^{2}
        -2\langle\nabla_{C_{1}}C_{2}, Y\rangle^{2}
        +\frac{2\beta^{2}}{\alpha^{2}+\beta^{2}}\left\{-(P_{1})^{2}+(P_{2})^{2}\right\}
        .\label{eq:206}
    \end{equation}
    By Lemma \ref{lemm:202}, $P_{1}=0$ or $P_{2}=0$ holds.
    If $P_{2}=0$, we have a contradiction by \eqref{eq:206}.
    Thus, we suppose that $P_{2}\neq 0$ and $P_{1}=0$.
    Using the Gauss equation, we have $\langle R(Y, X)Y, C_{1}\rangle=0$.
    By Lemmas \ref{lemm:101} and \ref{lemm:204}, we have
    \begin{equation}
        \langle R(Y, X)Y, C_{1}\rangle
        =
        -P_{2}\langle\nabla_{Y}C_{1}, C_{2}\rangle
        -Q_{1}\langle\nabla_{C_{1}}C_{1}, Y\rangle
        +P_{2}\langle\nabla_{C_{1}}C_{2}, Y\rangle
        .
        \label{eq:207}
    \end{equation}
    By \eqref{eq:3} of Lemma \ref{lemm:3} and Lemma \ref{lemm:201}, we have $-Q_{1}\langle\nabla_{C_{1}}C_{1}, Y\rangle+P_{2}\langle\nabla_{C_{1}}C_{2}, Y\rangle=-2P_{2}\langle\nabla_{Y}C_{1}, C_{2}\rangle$.
    Therefore, we have $\langle\nabla_{Y}C_{1}, C_{2}\rangle=0$, which is a contradiction by Lemma \ref{lemm:203}.

    In \ref{prop:4-4} of Proposition \ref{prop:4}, by Proposition \ref{prop:103} and Lemmas \ref{lemm:6}, \ref{lemm:201}, and \ref{lemm:203}, we have
    \begin{equation}
        0
        =
        -2\langle\nabla_{C_{1}}C_{1}, Y\rangle^{2}
        -2\langle\nabla_{C_{1}}C_{2}, Y\rangle^{2}
        +\frac{2\alpha^{2}}{\alpha^{2}+\beta^{2}}\left\{-(P_{1})^{2}+(P_{2})^{2}\right\}
        .\label{eq:208}
    \end{equation}
    By Lemma \ref{lemm:202}, $P_{1}=0$ or $P_{2}=0$ holds.
    If $P_{2}=0$, we have $\langle\nabla_{C_{1}}C_{2}, Y\rangle=0$ by \eqref{eq:208}, which is a contradiction by Lemma \ref{lemm:203}.
    If we suppose that $P_{2}\neq 0$ and $P_{1}=0$, we have a contradiction in a similar way to \ref{prop:4-3} of Proposition \ref{prop:4}.
\end{proof}

\subsubsection{Type IV}\label{sec:4-3-4}
\begin{theorem}\label{theo:ne-4}
    There exist no isoparametric hypersurfaces whose Petrov type is type IV of index 2 in $\mathbb{R}^{5}_{2}$ or $S^{5}_{2}$.
\end{theorem}
Suppose that there exists an isoparametric hypersurface $M$ whose Petrov type is type IV of index 2 in $\mathbb{R}^{5}_{2}$, $S^{5}_{2}$ or $S^{5}_{3}$.
By Proposition \ref{prop:1}, there exists a local frame field $C_{1}, C_{2}, X_{1}, X_{2}$ such that
\begin{equation*}
    AC_{1}=\alpha_{1} C_{1}+\beta_{1} C_{2},\quad
    AC_{2}=-\beta_{1} C_{1}+\alpha_{1} C_{2},\quad
    AX_{1}=bX_{1},\quad
    AX_{2}=bX_{2}+X_{1},
\end{equation*}
\begin{equation*}
    \langle C_{i}, C_{j}\rangle
    =
    \begin{cases}
        (-1)^{i}, & (i=j),\\
        0, & (\text{otherwise}),
    \end{cases}
    \quad
    \langle X_{i}, X_{j}\rangle
    =
    \begin{cases}
        1, & (i\neq j),\\
        0, & (\text{otherwise}),
    \end{cases}
    \quad
    \langle C_{i}, X_{1}\rangle
    =
    \langle C_{i}, X_{2}\rangle
    =
    0.
\end{equation*}
\begin{lemma}\label{lemm:104}
    We have $\langle\nabla_{X_{1}}X_{1}, C_{1}\rangle=\langle\nabla_{X_{1}}X_{1}, C_{2}\rangle=\langle\nabla_{X_{1}}X_{2}, C_{1}\rangle=\langle\nabla_{X_{1}}X_{2}, C_{2}\rangle=0$.
\end{lemma}
\begin{proof}
    In a similar way to Lemma \ref{lemm:101}, we have $\langle\nabla_{X_{1}}X_{1}, C_{1}\rangle=\langle\nabla_{X_{1}}X_{1}, C_{2}\rangle=0$ by $\{C_{1}, X_{1}\}X_{1}$ and $\{C_{2}, X_{1}\}X_{1}$.
    On the other hand, we have
    \begin{align}
        &\{C_{1}, X_{1}\}X_{2};
        &
        0&=
        -(\alpha-b)\langle\nabla_{X_{1}}X_{2}, C_{1}\rangle-\beta\langle \nabla_{X_{1}}X_{2}, C_{2}\rangle+\langle\nabla_{X_{1}}X_{1}, C_{1}\rangle,\label{eq:127}\\
        &\{C_{2}, X_{1}\}X_{2};
        &
        0&=
        -(\alpha-b)\langle\nabla_{X_{1}}X_{2}, C_{2}\rangle+\beta\langle \nabla_{X_{1}}X_{2}, C_{1}\rangle+\langle\nabla_{X_{1}}X_{1}, C_{2}\rangle
        .
        \label{eq:128}
    \end{align}
    We should solve the system of linear equations \eqref{eq:127} and \eqref{eq:128}.
\end{proof}
\begin{lemma}
    \label{lemm:105}
    For $Z=X_{1}, X_{2}$, we have
    \begin{equation*}
        \langle \nabla_{C_{1}}C_{1}, Z\rangle
        =
        \langle \nabla_{C_{2}}C_{2}, Z\rangle
        ,\quad
        \langle \nabla_{C_{1}}C_{2}, Z\rangle
        =
        -
        \langle \nabla_{C_{2}}C_{1}, Z\rangle.
    \end{equation*}
    And we have
    \begin{equation*}
        \langle \nabla_{C_{1}}C_{2}, X_{1}\rangle
        =
        \frac{-2\beta^{2}}{(\alpha-b)^{2}+\beta^{2}}
        \langle \nabla_{X_{1}}C_{1}, C_{2}\rangle
        .
    \end{equation*}
\end{lemma}
\begin{proof}
    In a similar way to Lemma \ref{lemm:3}, we have the above claim for $Z=X_{1}$ by $\{C_{1}, X_{1}\}C_{1}$, $\{C_{1}, X_{1}\}C_{2}$, $\{C_{2}, X_{1}\}C_{1}$ and $\{C_{2}, X_{1}\}C_{2}$.
    On the other hand, we have
    \begin{align}
        &\{C_{1}, X_{2}\}C_{1};
        &
        (\alpha-b)\langle \nabla_{C_{1}}C_{1}, X_{2}\rangle+\beta\langle \nabla_{C_{1}}C_{2}, X_{2}\rangle-\langle\nabla_{C_{1}}C_{1}, X_{1}\rangle&=-2\beta\langle \nabla_{X_{2}}C_{1}, C_{2}\rangle,\label{eq:129}\\
        &\{C_{1}, X_{2}\}C_{2};
        &
        (\alpha-b)\langle \nabla_{C_{1}}C_{2}, X_{2}\rangle-\beta\langle \nabla_{C_{1}}C_{1}, X_{2}\rangle-\langle\nabla_{C_{1}}C_{2}, X_{1}\rangle&=0,\label{eq:130}\\
        &\{C_{2}, X_{2}\}C_{1};
        &
        (\alpha-b)\langle \nabla_{C_{2}}C_{1}, X_{2}\rangle+\beta\langle \nabla_{C_{2}}C_{2}, X_{2}\rangle-\langle\nabla_{C_{2}}C_{1}, X_{1}\rangle&=0,\label{eq:131}\\
        &\{C_{2}, X_{2}\}C_{2};
        &
        (\alpha-b)\langle \nabla_{C_{2}}C_{2}, X_{2}\rangle-\beta\langle \nabla_{C_{2}}C_{1}, X_{2}\rangle-\langle\nabla_{C_{2}}C_{2}, X_{1}\rangle&=-2\beta\langle \nabla_{X_{2}}C_{1}, C_{2}\rangle.\label{eq:132}
    \end{align}
    We note that $\beta=0$.
    We should solve the system of linear equation \eqref{eq:129}, \eqref{eq:130}, \eqref{eq:131}, and \eqref{eq:132}.
\end{proof}
\begin{proposition}\label{prop:106}
    We have
    \begin{equation}
        \langle\nabla_{C_{1}}C_{1}, X_{1}\rangle=\langle\nabla_{C_{1}}C_{2}, X_{1}\rangle=\langle\nabla_{C_{2}}C_{1}, X_{1}\rangle=\langle\nabla_{C_{2}}C_{2}, X_{1}\rangle=\langle\nabla_{X_{1}}C_{1}, C_{2}\rangle=0
        .
        \label{eq:137}
    \end{equation}
\end{proposition}
\begin{proof}
    We note Remark \ref{rem:1}.
    By Lemmas \ref{lemm:2}, \ref{lemm:104}, and \ref{lemm:105}, and \eqref{eq:R}, we have
    \begin{align}
        \langle R(C_{1}, X_{1})C_{1}, X_{1}\rangle
        &=
        -\langle\nabla_{C_{1}}C_{1}, X_{1}\rangle^{2}-\langle\nabla_{C_{1}}C_{2}, X_{1}\rangle^{2}
        -
        \langle \nabla_{X_{1}}\nabla_{C_{1}}C_{1}, X_{1}\rangle
        .
        \label{eq:134}
    \end{align}

    Similarly, we have
    \begin{align}
        \langle R(C_{2}, X_{1})C_{2}, X_{1}\rangle
        &=
        \langle\nabla_{C_{1}}C_{1}, X_{1}\rangle^{2}+\langle\nabla_{C_{1}}C_{2}, X_{1}\rangle^{2}
        -
        \langle \nabla_{X_{1}}\nabla_{C_{1}}C_{1}, X_{1}\rangle
        .
        \label{eq:136}
    \end{align}
    By \eqref{eq:134} and \eqref{eq:136}, and using the Gauss equation, we have
    \begin{equation*}
        0
        =
        \langle R(C_{1}, X_{1})C_{1}, X_{1}\rangle -\langle R(C_{2}, X_{1})C_{2}, X_{1}\rangle
        =
        -2\langle\nabla_{C_{1}}C_{1}, X_{1}\rangle^{2}
        -2\langle\nabla_{C_{1}}C_{2}, X_{1}\rangle^{2}
        .
    \end{equation*}
    We have \eqref{eq:137} by Lemma \ref{lemm:105}.
\end{proof}
\begin{proposition}\label{prop:107}
    We have $\kappa+\nu(\alpha^{2}+\beta^{2})=0$.
\end{proposition}
\begin{proof}
    By computing $\langle R(C_{1}, C_{2})C_{1}, C_{2}\rangle$ by using Lemma \ref{lemm:2}, Proposition \ref{prop:106}, \eqref{eq:R}, and the Gauss equation.
\end{proof}
We can prove Theorem \ref{theo:ne-4} by Proposition \ref{prop:107}.
\begin{corollary}
    If $\lambda$ is a complex principal curvature of an isoparametric hypersurface whose Petrov type is type IV of index 2 in $S^{5}_{3}$, then $|\lambda|=1$ holds.
\end{corollary}

\appendix
\section{Non-existence of Index 1}\label{sec:A}
We can prove the claim of the non-existence of index 1 by a similar technique in Section \ref{sec:4-3}.
\begin{theorem}
    \label{theo:ne-5}
    There exist no isoparametric hypersurfaces whose Petrov type is type IV of index 1 in $\mathbb{R}^{n+1}_{1}$, $\mathbb{R}^{n+1}_{2}$, $S^{n+1}_{1}$, or $H^{n+1}_{2}$.
\end{theorem}
\begin{proof}
    In the case where $\kappa=0$ and $\nu=1$, and where $\kappa=1$ and $\nu=1$, we have already proven by \cite[Theorem 4.10]{MR783023} and \cite[Theorem 4.1]{MR3077208}, respectively.
    We give a proof in the other cases by contradiction.
    Suppose that there exists an isoparametric hypersurface $M$ whose Petrov type is type IV of index 1 in $\mathbb{R}^{n+1}_{2}$ or $H^{n+1}_{2}$.
    By Proposition \ref{prop:1}, there exists a local frame field $C_{1}, C_{2}, Y_{1}, \ldots, Y_{n-2}$ such that
    \begin{equation*}
        AC_{1}=\alpha C_{1}+\beta C_{2},\quad
        AC_{2}=-\beta C_{1}+\alpha C_{2},\quad
        AY_{k}=a_{k}Y_{k},
    \end{equation*}
    \begin{align*}
        \langle C_{i}, C_{j}\rangle
        =
        \begin{cases}
            (-1)^{i}, & (i=j),\\
            0, & (\text{otherwise}),
        \end{cases}
        \quad
        \langle Y_{k}, Y_{l}\rangle
        =
        \begin{cases}
            1, & (k=l),\\
            0, & (\text{otherwise}),
        \end{cases}
        \quad
        \langle C_{i}, Y_{k}\rangle
        =
        0.
    \end{align*}
    Using the Gauss equation, we have
    \begin{equation}
        \langle R(C_{1}, C_{2})C_{1}, C_{2}\rangle
        =
        \kappa+\nu(\alpha^{2}+\beta^{2})
        <
        0
        .
        \label{eq:60}
    \end{equation}
    Lemma \ref{lemm:2} is a claim in Section \ref{sec:4-3-1}.
    However, Lemma \ref{lemm:2} also holds in Appendix \ref{sec:A}.
    By Lemma \ref{lemm:2} and \eqref{eq:R}, the right-hand side of \eqref{eq:60} is equal to
    \begin{equation*}
        \sum_{k}\langle \nabla_{C_{1}}C_{2}, Y_{k}\rangle^{2}
        +
        \sum_{k}\langle \nabla_{C_{1}}C_{1}, Y_{k}\rangle^{2}
        +
        \sum_{k}
        \frac{4\beta^{2}}{(\alpha-a_{k})^{2}+\beta^{2}}\langle\nabla_{Y_{k}}C_{1}, C_{2}\rangle^{2}
        \geqq 0
        ,
    \end{equation*}
    which is a contradiction.
\end{proof}

\end{document}